\documentclass{article}

\usepackage[nonatbib,preprint]{neurips_2026}
\usepackage[imports,opt,theoremstyles]{commands}
\usepackage{xcolor}
\usepackage{booktabs}
\usepackage{wrapfig}

\newcommand{\Bc}{\mathcal B}
\newcommand{\Bs}{\mathcal B^*}
\newcommand{\dualnorm}[1]{\left\lVert #1 \right\rVert_*}
\newcommand{\nucnorm}[1]{\norm{#1}_{\operatorname{S_1}}}
\newcommand{\opnorm}[1]{\norm{#1}_{\operatorname{S_\infty}}}
\newcommand{\fnorm}[1]{\norm{#1}_{\operatorname{F}}}
\newcommand{\angles}[1]{\left\langle #1 \right\rangle}

\newcommand{\euclvar}{\sigma_F}

\newcommand{\nameeqref}[1]{\hyperref[#1]{(\nameref{#1})}}
\newcommand{\lsmooth}{\Cref{assum:lsmooth}}

\newcommand{\pBCM}{\Cref{assum:noise}}
\newcommand{\ourAlgName}{\textsc{Nssd-M}}
\newcommand{\ourAlg}{\hyperref[eq:snsdm]{\ourAlgName}}
\newcommand{\muon}{\textsc{Muon}}
\newcommand{\lmo}{\operatorname{lmo}}

\usepackage{subcaption}
\usepackage[sorting=ynt,
			natbib=true,
			style=authoryear-comp,
			maxcitenames=2,
			uniquelist=false,
			backend=biber]{biblatex}
\usepackage{changes}

\title{Free Heavy-Tailed Lunch for Muon: \\A Theoretical Justification of Empirical Success}

\newcommand{\affilfont}{\normalfont\mdseries\fontsize{8}{9.5}\selectfont}
\author{%
\begin{tabular}{@{}c@{}}
\phantom{a}\\
\textbf{Florian Hübler}\textsuperscript{1,2,\thanks{Corresponding author:
\href{mailto:florian.huebler@inf.ethz.ch}{florian.huebler@inf.ethz.ch}}}
\qquad
\textbf{Thomas Pethick}
\qquad
\textbf{Suvrit Sra}\textsuperscript{2,3}
\\[6mm]
{\affilfont \textsuperscript{1}Department of Computer Science, ETH Zurich, Switzerland}
\\
{\affilfont \textsuperscript{2}Department of Mathematics, Technical University of Munich, Germany}
\\
{\affilfont \textsuperscript{3}Munich Center for Machine Learning (MCML)}
\vspace*{-2mm}
\end{tabular}%
}

\date{}

\begin{document}

\maketitle

\begin{abstract}
Non-Euclidean optimisation methods with matrix-valued updates, such as \textsc{Muon} and \textsc{Scion}, have recently shown strong empirical performance for training Transformer models, yet their theoretical advantages over Euclidean methods remain poorly understood. We address this gap in the heavy-tailed non-convex regime, where stochastic gradients have bounded $p$-th central moments, $p \in (1,2]$. We show that certain non-Euclidean methods achieve optimal sample complexity under stronger stationarity measures, while Euclidean methods incur additional dimension-dependent costs. As a consequence, for $m \times n$ matrices, \textsc{Muon} finds an $\varepsilon$-stationary point in nuclear norm within $\mathcal{O}\left(\min\{m, n\} \frac{\Delta_1 L}{\varepsilon^2} \left(\frac \sigma \varepsilon \right)^{\frac p {p-1}}\right)$ samples, absorbing heavy-tailed noise without extra dimension dependence, unlike Euclidean methods. We further prove this sample complexity, including its dimension dependence, is optimal for all first-order methods under nuclear-norm stationarity. Experiments on large language models support our theory. Surprisingly, our results suggest that other Schatten geometries beyond the spectral geometry of \textsc{Muon} can perform competitively in certain settings.
\end{abstract}

\section{Introduction}
\label{sec:intro}

We consider the unconstrained stochastic optimisation problem
\begin{equation}\label{eq:stoch_minimisation}
    \min_{x \in \Bc} \set{F(x) \coloneqq \Exp{f(x, \xi)}},
\end{equation}
where $\xi$ is distributed according to some unknown distribution and $\Bc$ is a Banach space such as $\R^d$ or $\R^{m \times n}$. This setting has been the focus of a large body of work in optimisation due to its prevalence in machine learning and data-driven optimisation \citep{Bottou2018}.

While \textsc{Adam} \citep{Adam2014Kingman} has been the de-facto standard optimisation algorithm for training Transformer-based models for the last decade, a recent line of literature has proposed a new class of methods based on steepest descent in general normed spaces that are empirically competitive with \textsc{Adam} and its variants. The general form of these algorithms is given by the normalised steepest descent \citep{ConvexOptimization2004boyd} update
\begin{equation}\label{eq:intro_nssd}
    \xtp \gets \xt + \eta \argmin_{\norm{d} \leq 1} \angles{\gt, d},
\end{equation}
where $\ssi > 0$ is the stepsize, $\gt$ some gradient estimator and $\norm \cdot$ a norm on $\Bc$ \citep{SCION2025pethick}. The most prominent example is \muon\ \citep{MuonOptimizerHidden2024jordan,OrthogonalisingGradientsSpeed2022tuddenham}, which uses (1) the spectral norm $\norm \cdot = \opnorm \cdot$, (2) the momentum gradient estimator, and (3) Newton-Schulz iterations \citep{NewtonSchulz1970Kovarik} to estimate the $\argmin$. This variant achieves impressive results on the modded nanoGPT speedrun \citep{NanoGPTSpeedrun2024Jordan} and was used to train recent trillion parameter language models such as Kimi K2 \citep{team2025kimi} and DeepSeek V4 \citep{deepseekai2026deepseekv4}.

From a theoretical perspective, recent works have established convergence of such algorithms in the non-convex setting. These guarantees are typically of the form
\begin{align}\label{eq:dimension_dep_guarantee}
    \Exp{\dualnorm{\nF \pare{{\xli}}}} \leq \eps \quad \text{after} \quad T = \Oc \pare{\rho^2 \cdot \frac{\Dz L \sigma^2}{\eps^4}} \quad \text{samples,}
\end{align}
where $F$ is $L$-smooth and the stochastic gradients have bounded variance $\sigma^2$, $\dualnorm \cdot$ is the dual norm and $\rho \coloneqq \max_{v \in \Bs} \nicefrac{\dualnorm{v}}{\norm{v}_{*,2}}$ is the norm-equivalence constant between the dual norm and a Hilbert norm on $\Bs$ \citep[e.g.,][]{ConvergenceAnalysisMuon2025shen,SCION2025pethick}. Importantly, \eqref{eq:dimension_dep_guarantee} guarantees $\Exp{\dualnorm{\nF \pare{\xli}}} \leq \eps$, which we denote $\eps$-$\dualnorm \cdot$-stationary point. For example, for \muon\ the dual norm is the Schatten-1 norm $\nucnorm \cdot$ and thus a strictly stronger measure than the Frobenius norm typically used in Euclidean analyses. While informative, these guarantees face two difficulties in capturing the potential benefits of non-Euclidean methods in stochastic settings.

First, the appearance of the norm-equivalence constant $\rho$ can make the bound effectively dimension-dependent, thereby obscuring potential advantages over Euclidean methods.
This term arises from the treatment of the gradient noise. Indeed, the analysis must control terms of the form $\E[\dualnorm{\sum \pare{\nFxt - \nfxt}}]$. Such bounds are straightforward in Hilbert spaces, where the inner-product structure is available, but much harder for general norms. Existing analyses therefore use the norm equivalence $\dualnorm \cdot \leq \rho \norm{\cdot}_{*,2}$, before controlling the noise in the Euclidean geometry. In particular, for \muon's nuclear norm on $\R^{m \times n}$, we get $\rho = \sqrt{\min\set{m, n}}$.

Second, these results rely on the bounded variance assumption, which can be overly restrictive. Indeed, empirical evidence from image classification \citep{TailIndexAnalysisStochastic2019simsekli,RevisitingNoiseModel2024Battash}, language modelling \citep{WhyAreAdaptive2020zhang,ahn2024linear}, and reinforcement learning (RL) \citep{ProximalPolicyOptimizations2021garg} suggests that stochastic gradients can exhibit heavy-tailed behaviour. 
These observations suggest that heavy-tailed noise provides a more faithful model of stochastic gradients in modern machine learning. 
Consequently, if one seeks to understand the empirical advantage of non-Euclidean optimisation methods, it is natural to analyse them in this regime. These two limitations naturally raise the following question:
\begin{center}
\emph{Can non-Euclidean methods such as \muon\ admit provable improvements over Euclidean methods in stochastic regimes, particularly under heavy-tailed noise?}
\end{center}

To study this question, we depart from the usual Euclidean noise assumption and instead make the heavy-tailed assumption directly in the native geometry: for some tail index $p \in (1,2]$ we assume
\begin{equation*}
    \Exp{\dualnorm{\nf \pare{x, \xi} - \nFx}^p} \leq \sigma^p.
\end{equation*}
This choice is the root of the main technical challenges, but also what enables sharper, geometry-aware guarantees.

\paragraph{Contributions.}
In this work we show that non-Euclidean methods can achieve strict improvements over their Euclidean counterparts in the heavy-tailed setting.

\begin{enumerate}[label=(\roman*)]
    \item We show that normalised steepest descent methods with momentum can achieve optimal rates for strictly stronger convergence measures. More precisely, we provide the optimal sample complexity bound
    $
        \Oc \pare{\frac{\Dz L}{\eps^2} \pare{\frac \sigma \eps}^{\frac{p}{p-1}}}
    $
    for reaching an $\eps$-$\dualnorm \cdot$-stationary point, whenever the dual norm $\dualnorm \cdot$ is $p$-uniformly smooth (see \Cref{def:constants}). Moreover, we provide this guarantee under the weak $(L_0, L_1)$-smoothness assumption.
    \item Based on this result, we obtain improved convergence guarantees for \muon, whose dual norm $\nucnorm \cdot$ is not uniformly smooth. We show that \muon\ achieves the sample complexity
    \begin{equation*}
        \Oc \pare{\min\set{m,n} \cdot \frac{\Dz L}{\eps^2} \pare{\frac{\varsym}{\eps}}^{\frac{p}{p-1}}} \qquad \text{vs.} \qquad
        \Oc \pare{\min\set{m,n}^{\frac{3p-2}{2(p-1)}} \cdot \frac{\Dz L}{\eps^2} \pare{\frac{\varsym}{\eps}}^{\frac{p}{p-1}}}
    \end{equation*}
    obtained by Euclidean algorithms when corrected for the stationarity measure. This shows that \muon\ can absorb heavy-tailed noise essentially without extra dimensional cost, whereas classical Euclidean methods can match the same stationarity measure only at the price of a strictly worse, $p$-dependent dimension factor.
    \item We provide a first-order sample complexity lower-bound of
    \begin{equation*}
        \Omega \pare{\min\set{m,n} \cdot \frac{\Dz L}{\eps^2} \pare{\frac{\varsym}{\eps}}^{\frac{p}{p-1}}}
    \end{equation*}
    for $\nucnorm \cdot$-stationarity. In particular, \muon\ is first-order optimal and there cannot exist a dimension-free guarantee for $\nucnorm \cdot$-stationarity.
\end{enumerate}
As a side result, we provide tools from functional analysis to control the gradient noise for general norms, sidestepping the usual reduction to Euclidean martingale bounds through norm equivalence. These tools may be of independent interest for the analysis of general-norm algorithms beyond the heavy-tailed setting.

\section{Preliminaries}
\label{sec:prelims}

\textbf{Notation.} Throughout, $(\Bc, \norm{\cdot})$ denotes a Banach space with dual $(\Bs, \dualnorm{\cdot})$. We denote the duality pairing by $\angles{x^*, x} \coloneqq x^*(x)$ for $x \in \Bc, x^* \in \Bs$. To keep consistency with the typical notation in $\R^d$ we abuse notation\footnote{Note that the gradient is in general not well defined in Banach spaces, as it is defined by the Riesz representation of $DF(x)$, which only exists in Hilbert spaces.} by writing $\nabla F(x)$ for the Frechet derivative $D F(x) \in \Bs$ and $\angles{\nabla F(x), d}$ for the duality pairing between the Frechet derivative and $d \in \Bc$.

\textbf{Algorithm.} We study the stochastic normalised steepest descent method with momentum, 
\begin{align}\tag{\ourAlgName}\begin{split}\label{eq:snsdm}
    m_t &\gets \momt m_{t-1} + (1-\momt) \nfxt, \\
    x_{t+1} &\gets x_t + \sst \argmin_{\norm{d} \leq 1} \angles{m_t, d},
\end{split}\end{align}
where $\iterationxi \fin, \iterationxi 2, \dots \iid \xi$ are independent copies of $\xi$, $\nfxt$ is a stochastic gradient oracle of $\nabla F(x_t), \beta \in [0,1)$ is the momentum parameter and $m_0 \coloneqq \nabla f(\xfin, \iterationxi \fin)$. When choosing $(\Bc, \norm{\cdot}) = (\R^{m \times n}, \opnorm \cdot)$ this corresponds to an idealised version of \muon\ with an exact lmo, and  $(\Bc, \norm{\cdot}) = (\R^d, \norm \cdot_{\ell_\infty})$ to \textsc{Signum}. We consider the following standard assumptions.

\begin{assum}
\label{assum:lower}
There exists $\Dz < \infty$ such that $F(x_1) - \inf_{x \in \Bc} F(x) \leq \Dz$.
\end{assum}

For ease of exposition, we present the results for the case of $L$-smoothness as defined below. Our results readily also hold for the weaker notion of $(L_0, L_1)$-smoothness and can be found in \Cref{sec:app.proofs}.

\begin{assum}
\label{assum:lsmooth}
There exists $L \geq 0$ such that for all $x,y \in \Bc$, $\dualnorm{\nabla F(x) - \nabla F(y)} \leq L \norm{x-y}$.
\end{assum}

Finally, as mentioned in the introduction, we consider the heavy-tailed setting where the stochastic gradient oracle has bounded $p$-th central moment for some $p \in (1, 2]$.

\begin{assum}
\label{assum:noise}
There exists $p \in (1, 2], \varsym \geq 0$ such that the stochastic gradient oracle is unbiased and has bounded $p$-th central moment, i.e., for all $x \in \Bc$,
\begin{align*}
    \Exp{\nf\pare{x, \xi}} = \nFx, \qquad \text{and} \qquad
    \Exp{\dualnorm{\nf\pare{x, \xi} - \nFx}^p} \leq \varsym^p.
\end{align*}
\end{assum}

\subsection{Moment Bounds for General Norms}

We first summarise the necessary technical tools to control the gradient approximation error for general norms, sidestepping the usual reduction to Euclidean martingale bounds through norm equivalence. Therefore we first recall the following key geometric definitions of general norms.

\begin{definition}[{\cite[Equation (2.7)]{SharpUniformConvexity1994ball}}]\label{def:constants}
    Let $(X, \norm \cdot)$ be a Banach space.
    \begin{enumerate}[label=\roman*)]
        \item The \emph{$p$-smoothness constant} of $(X, \norm \cdot)$ is defined as
        \begin{equation*}
            S_p(X) \coloneqq \inf\set{S \geq 1 \suchthat \fas x,y \in X\colon \norm{x+y}^p + \norm{x-y}^p \leq 2\norm{x}^p + 2 S^p \norm{y}^p},
        \end{equation*}
        and we call $(X, \norm \cdot)$ $p$-uniformly smooth if $S_p(X) < \infty$.
        \item The \emph{$q$-convexity constant} of $(X, \norm \cdot)$ is defined as
        \begin{equation*}
            K_q(X) \coloneqq \inf\set{K \geq 1 \suchthat \fas x,y \in X\colon 2\norm{x}^q + \nicefrac 2 {K^q}\norm y^q \leq \norm{x+y}^q + \norm{x-y}^q },
        \end{equation*}
        and we call $(X, \norm \cdot)$ $q$-uniformly convex if $K_q(X) < \infty$.
    \end{enumerate}
\end{definition}

The following strong duality holds between the convexity and smoothness constants.

\begin{lemma}[{\cite[Lemma 5]{SharpUniformConvexity1994ball}}]\label{lem:convexity_smoothness_duality}
    Let $(X, \norm \cdot)$ be a Banach space and $1 < p \leq 2 \leq q < \infty$ conjugate exponents, i.e., $\nicefrac 1 p + \nicefrac 1 q = 1$. Then $S_p(X^*) = K_q(X)$.
\end{lemma}

The above duality relationship allows us to verify the smoothness assumption on the dual space by checking a convexity property of the primal space, which is often more natural. Next let us get some intuition of the above definition for the special cases of $\R^d$ and $\R^{m \times n}$

\begin{example}[{\cite{NonCommutativeClarksonInequalities2002hirzallah}}]
    \label{example:schatten_p}
    Consider the $\ell_r$ space $\pare{X, \norm \cdot} = (\mathbb{R}^d, \norm{\cdot}_r)$ or the Schatten-$r$ space $(X, \norm \cdot) = (\R^{m \times n}, \norm \cdot_{S_r})$. In both cases the following holds.
    \begin{enumerate}[label=\roman*)]
        \item If $1 < r \le 2$, then $X$ is $r$-uniformly smooth with $S_r(X)=1$.
        \item If $2 \le r < \infty$, then $X$ is $r$-uniformly convex with $K_r(X) = 1$.
    \end{enumerate}
\end{example}

Finally we provide the main technical tool which will allow us to control the gradient approximation error for general norms, sidestepping the usual reduction to Euclidean martingale bounds through norm equivalence.

\begin{theorem}[{\cite[c.f.,][Theorem 3.2]{StrongMartingaleType2005wenzel}}]
\label{thm:banach_martingale}
Every martingale difference sequence $Y_1, \dots, Y_n$ in $\Bs$ satisfies
\begin{equation*}
    \Exp{\dualnorm{\sum_{t=1}^n Y_t}^p}
    \leq
    S_p(\Bs)^p \cdot \sum_{t=1}^n \Exp{\dualnorm{Y_t}^p}.
\end{equation*}
\end{theorem}

We note that the above inequality is vacuous if $\Bs$ is not $p$-uniformly smooth, i.e., if $S_p(\Bs) = \infty$.
\section{Main Result}
\label{sec:main}

In this section we present our theoretical results. In \Cref{sec:main.general_norms} we provide convergence guarantees for \ourAlg\ for uniformly convex norms under heavy-tailed noise. In \Cref{sec:main.muon} we extend these results to the case of \muon. Finally, in \Cref{sec:main.lb} we provide a lower bound showing that the obtained dimension dependence is first-order optimal for nuclear-stationarity guarantees.

\subsection{Guarantees for Uniformly Convex Norms}
\label{sec:main.general_norms}

We are now ready to prove the convergence guarantee for uniformly convex norms under heavy-tailed noise. For ease of exposition we will present it under the standard $L$-smoothness assumption, extended results for $(L_0, L_1)$-smoothness and the corresponding proofs can be found in \Cref{sec:app.upper_bounds.general}.

\begin{theorem}[Uniformly-Convex Norm Guarantee]\label{thm:main_conv}
    Suppose that \Cref{assum:lower,assum:lsmooth,assum:noise} hold. Furthermore let $q = \nicefrac p {(p-1)}$ and assume that $(\Bc, \norm \cdot)$ is $q$-uniformly convex with $q$-convexity constant $K$. Then the iterates generated by \ourAlg\ with parameters
    \begin{align*}
        \sst &\equiv \sqrt{\frac{\Dz (1 - \momentum)}{L T}}, \qquad
        \momt \equiv \momentum = 1 - \min\set{1, \max\set{T^{-\frac p {2p-1}}, \pare{\frac{\Dz L}{\varsym^2 K^2 T}}^{\frac p {3p-2}}}}
    \end{align*}
    satisfy
    \begin{equation*}
        \frac 1 \li \iSum \Exp{\dualnorm{\nFxt}}
        \leq 9 \sqrt{\frac{\Dz L} \li} 
        + 14 \pare{\frac{\Dz L } \li \pare{\varsym K}^{\frac p {p-1}}}^{\frac {p-1}{3p-2}}
        + 10 \frac{\varsym K}{\li^{\frac{p-1}{2p-1}}}.
    \end{equation*}
\end{theorem}

This shows that \ourAlg\ achieves the optimal heavy-tailed sample complexity in arbitrary, possibly infinite-dimensional Banach spaces. Importantly, the stationary measure $\dualnorm \cdot$ can in general be stronger than in the Euclidean case. This becomes particularly transparent for Schatten and $\ell$ norms.

\begin{corollary}\label{cor:main_conv_Schatten_ell}
    Let $(\Bc,\norm{\cdot}) = (\R^{m \times n}, \norm{\cdot}_{S_r})$ or $(\R^d, \norm{\cdot}_{\ell_r})$, where $r \leq \nicefrac p {(p-1)}$. Then, in the setting of \Cref{thm:main_conv}, \ourAlg\ satisfies
    \begin{equation}\label{eq:lmo_sample_guarantee}
        \Exp{\dualnorm{\nF \pare{\iterationx \tau}}} \leq \eps 
        \quad \text{after} \quad
        T = \Oc \pare{\frac{\Dz L} {\eps^2} \pare{\frac \sigma \eps}^{\frac{p}{p-1}}} 
        \quad \text{samples,} 
    \end{equation}
    where $\tau \sim \unif([\li])$. Here $\dualnorm \cdot = \norm\cdot_{S_{r'}}$ and $\norm \cdot_{\ell_{r'}}$, where $r' = \nicefrac r {(r-1)}$, for the Schatten and $\ell$ case respectively.
\end{corollary}

In particular, when choosing $r = 2$, \Cref{cor:main_conv_Schatten_ell} exactly reconstructs the existing heavy-tailed results for the Euclidean case \citep{NonconvexStochasticOptimization2025Liu,GradientClippingNormalization2024hubler}.
For $2 < r \leq \nicefrac{p}{(p-1)}$, the result is strictly stronger, as $\dualnorm \cdot \geq \norm{\cdot}_2 \geq \nicefrac 1 \rho \dualnorm \cdot$, where $\rho = \min\set{m, n}^{\nicefrac{(r-2)}{2r}}$ in Schatten-, and $\rho = d^{\nicefrac{(r-2)}{2r}}$ in $\ell$-geometry. A stationary-measure comparable result from the Euclidean setting hence has the form 
\begin{equation*}
    \Exp{\dualnorm{\nF \pare{\iterationx \tau}}} \leq \eps 
    \quad \text{after} \quad
    T = \Oc \pare{\rho^{\frac{3p-2}{p-1}} \cdot \frac{\Dz L} {\eps^2} \pare{\frac \sigma \eps}^{\frac{p}{p-1}}} 
    \quad \text{samples.}
\end{equation*}
This multiplicative factor can be considerable, as already for moderate values such as $p = 1.5$ and $r = \nicefrac{p}{(p-1)}$ we have $\rho^{\nicefrac{(3p-2)}{(p-1)}} \approx 10^4$ in the Schatten-, and $\approx 10^7$ in the $\ell$-geometry for modern transformers \citep{PalmScalingLanguage2023chowdhery}.

The main technical novelty in the proof of \Cref{thm:main_conv} is based on the careful application of \Cref{thm:banach_martingale} to derive the following deviation bound on our gradient estimator.

\begin{lemma}\label{lem:mom:deviation_bound_simplified}
    Suppose that \Cref{assum:lsmooth,assum:noise} hold, and let $q = \nicefrac p {(p-1)}$. Assume that $(\Bc, \norm \cdot)$ is $q$-uniformly convex with $q$-convexity constant $K$. Then the iterates generated by \ourAlg\ with $\sst \equiv \ssi$ and $\momt \equiv \beta$ satisfy
    \begin{equation*}
        \iSum \Exp{\dualnorm{\nFxt - \mt}}
        \leq
        \frac{ \varsym K }{1-\beta}
        + \varsym K T (1-\beta)^{\frac{p-1}{p}}
        + \frac{\ssi L T}{1-\beta}.
    \end{equation*}
\end{lemma}

The above Lemma allows us to control the convergence of \ourAlg\ directly in its native geometry, without having to reduce to the Euclidean case through norm equivalence, thereby avoiding the dimension dependence. In the special case of the Euclidean geometry, the above Lemma exactly reconstructs previous results \citep{MomentumImprovesNormalized2020cutkosky}. The extended version of this Lemma and its proof can be found in \Cref{lem:mom:deviation_bound}, we provide a proof sketch below.

\begin{proof}[Proof Sketch of \Cref{lem:mom:deviation_bound_simplified}]
    First note that the momentum deviation can be unrolled to
    \begin{equation*}
        \mt - \nFxt
        = a (\iterationm \fin - \nF \pare{\xfin})
        + \sum_{\tau = 2}^\ii b_\tau \pare{\nf\pare{\iterationx \tau, \iterationxi \tau} - \nF \pare{\iterationx \tau}}
        + \sum_{\tau = 2}^\ii c_\tau \pare{\nF \pare{\iterationx {\tau - 1}} - \nF \pare{\iterationx \tau}},
    \end{equation*}
    for some $a, b_\tau, c_\tau$. Here the last term can be controlled deterministically using smoothness, and the first term can be controlled after taking expectation. For the middle term let us denote $\gat \coloneqq \nfxt - \nFxt$ and note that $\iterationga \fin, \dots, \iterationga \ii$ is a martingale difference sequence in $\Bs$. Furthermore, by \Cref{lem:convexity_smoothness_duality}, we know that $\Bs$ is $p$-uniformly smooth with $S_p(\Bs) = K$. Hence we can apply \Cref{thm:banach_martingale} to get
    \begin{align*}
        \Exp{\dualnorm{\sum_{\tau = 2}^\ii b_\tau \pare{\nf\pare{\iterationx \tau, \iterationxi \tau} - \nF \pare{\iterationx \tau}}}}
        &\leq \Exp{\dualnorm{\sum_{\tau = 1}^\ii b_\tau \pare{\nf\pare{\iterationx \tau, \iterationxi \tau} - \nF \pare{\iterationx \tau}}}^p}^{\nicefrac 1 p}\\
        &\leq K \pare{\sum_{\tau = 1}^\ii b_\tau^p \Exp{\dualnorm{\iterationga \tau}^p}}^{\nicefrac 1 p}
        \leq \sigma K \pare{\sum_{\tau = 1} b_\tau^p}^{\nicefrac 1 p}.
    \end{align*}
    Plugging in the corresponding values of $a, b_\tau, c_\tau$ wraps up the proof.
\end{proof}

\begin{remark}
    It is well-known that weight-decay is often required in practice to efficiently train large language models \citep{team2025kimi}. 
    For normalised steepest descent based methods such as \ourAlg, this regularisation has a natural constrained interpretation as a conditional gradient step for optimisation over the corresponding norm ball \citep{SCION2025pethick}.
    We extend our convergence analysis to this constrained stochastic conditional gradient setting in \Cref{sec:app.proofs.scg}.
\end{remark}

\subsection{Guarantees for \texorpdfstring{\muon}{Muon}}
\label{sec:main.muon}
While the previous results established convergence of \ourAlg\ for a broad range of norms, they do not cover \muon\ as the Schatten-$\infty$ norm is not uniformly convex. Instead, we will show in this section that \muon\ can automatically adapt to the optimal Schatten-$r$ geometry, thereby achieving a reduced dimension dependence without any modification to the algorithm. A more general version under $(L_0, L_1)$-smoothness and the corresponding proofs can be found in \Cref{sec:app.proofs.muon}.

The main technical idea is to use equivalence of norms only to get from the nuclear norm to the Schatten-$p$ norm, where $p$ is the tail index, instead of going all the way to the Frobenius norm. This allows us to heavily reduce the dimension dependence as the noise becomes more heavy-tailed, while still controlling a strictly stronger stationarity measure than the Frobenius norm. Let us now state the result.

\begin{theorem}[Convergence of \muon\ under Heavy-Tailed Noise]
    \label{thm:muon_convergence}
    Consider $\Bc = \R^{m \times n}$ and suppose \Cref{assum:lower,assum:lsmooth,assum:noise} hold. Furthermore let $\kappa \coloneqq \min\set{m, n}^{\nicefrac{(p-1)}p}$. Then the iterates generated by \muon\ with parameters
    \begin{equation*}
        \sst \equiv \sqrt{\frac{\Dz (1 - \momentum)}{L T}}, \qquad
        \momt \equiv 1 - \min\set{1, \max\set{
            \pare{\frac{\Dz L}{\varsym^2 \kappa^2 T}}^{\frac{p}{3p-2}},
            \li^{-\frac{p}{2p-1}}
        }}
    \end{equation*}
    satisfy
    \begin{equation*}
        \frac 1 \li \iSum \Exp{\nucnorm{\nFxt}}
        \leq 9 \sqrt{\frac{\Dz L} \li} 
        + 14 \pare{\min\set{m, n} \frac{\Dz L \varsym^{\frac p {p-1}}} \li}^{\frac{p-1}{3p-2}} 
        + 10 \frac{\min\set{m, n}^{\frac{(p-1)}p} \varsym }{\li^{\frac{p-1}{2p-1}}}.
    \end{equation*}
\end{theorem}

When dropping non-leading terms, the sample complexity guarantee for \muon\ is hence given by
\begin{equation*}
    \Oc\pare{
        \min\set{m,n} \cdot \frac{\Dz L}{\eps^2}\pare{\frac{\varsym} \eps}^{\frac p {p-1}}
    }.
\end{equation*}
In comparison, the existing heavy-tailed Euclidean guarantees \citep{GradientClippingNormalization2024hubler,NonconvexStochasticOptimization2025Liu} imply a nuclear norm sample complexity guarantee of
\begin{equation}\label{eq:euclidean_stationarity_adjusted_guarantee}
    \Oc \pare{
        \min\set{m,n}^{\frac{3p-2}{2(p-1)}} \cdot \frac{\Dz L}{\eps^2}\pare{\frac{\varsym} \eps}^{\frac p {p-1}}
    },
\end{equation}
which is strictly worse in its dimension dependence. To illustrate, for $p = 1.5$, modern transformer based models \citep{PalmScalingLanguage2023chowdhery} satisfy $\min\set{m,n} \approx 10^4$ but $\min\set{m,n}^{\frac{3p-2}{2(p-1)}} \approx 10^{10}$, reducing the dimension dependence by 6 orders of magnitude. Hence, \Cref{thm:muon_convergence} shows that \muon\ can automatically adapt to the optimal Schatten-$p$ geometry, thereby avoiding the worst-case rank-dependence of existing heavy-tailed guarantees using the Euclidean geometry.

\begin{remark}\label{rem:full_network_muon}
For simplicity of presentation, 
\Cref{thm:muon_convergence} is stated for a single matrix. However, the extended result \Cref{thm:convergence_nssdm_norm_equivalence_formulation} directly applies to neural networks consisting of a collection of weight matrices $\{W_l\}_{l\in[D]}$, by equipping $\R^{m_1, n_1} \times \hdots \times \R^{m_D, n_D}$ with the norm $\max_{l\in[D]} \|W_l\|_{S_\infty}$ \citep{MUON2024Bernstein}. In this case, the dimension dependence becomes ${\sum_{l \in [D]} \min\set{m_l, n_l}}$. A further extension to full networks including vector-valued weights is discussed in \Cref{rem:application_full_network_including_vectors}.
\end{remark}

\subsection{First-Order Lower-Bounds for Nuclear-Stationarity}
\label{sec:main.lb}

Finally we theoretically examine whether the remaining dimension-dependence in the guarantee for \muon\ can be avoided. We show that, in the above setting, \emph{no first-order algorithm} can achieve a better dimension-dependence than \muon, when convergence is measured in $\nucnorm \cdot$.

We make use of the sample complexity framework of \citet{LowerBoundsNonConvex2023arjevani}. More precisely, we consider possibly randomised first-order algorithms $A \in \Ac_{\text{rand}}$ which, given a starting point $x_0$, random seed $r$ and access to a stochastic gradient oracle $\nf \pare{\cdot, \xi}$, produce a sequence of random iterates $\xfin, \iterationx 2, \dots$. For $\Dz, L \geq 0$, we consider the function class
\begin{equation*}
    \Fc_{\Dz, L}(m, n) \coloneqq \big\{
        F \colon \R^{m \times n} \to \R\, \big\vert
        \dualnorm{\nabla F(x) - \nabla F(y)} \leq L \norm{x-y} \text{and } F(0) - \inf_x F(x) \leq \Dz
    \big\}
\end{equation*}
and, for $\sigma \geq 0, p \in (1, 2]$ the oracle class $\Oc_{p, \sigma}$ of gradient oracles satisfying \pBCM. We will derive lower-bounds on the distributional complexity
\newcommand{\randCompl}{\mathfrak m_{\eps}^{\text{rand}}\pare{\dualnorm \cdot, \Dz, L, p, \varsym}}
\newcommand{\nucCompl}{\mathfrak m_{\eps}^{\text{rand}}\pare{\nucnorm \cdot, \Dz, L, p, \varsym}}
\begin{equation*}
    \randCompl \coloneqq 
    \sup_{O \in \Oc_{p, \varsym}} 
    \sup_{P_F \in \mathcal P [\Fc_{\Dz, L}(m, n)]}
    \inf_{A \in \Ac_{\text{rand}}}
    \inf \set{T \in \N \suchthat \Exp{\dualnorm{\nabla F(x_T)}} \leq \eps},
\end{equation*}
where $\mathcal P [\Fc_{\Dz, L}(m, n)]$ denotes the set of all probability distributions over $\Fc_{\Dz, L}(m, n)$. Lower-bounds on this distributional complexity imply lower-bounds on the worst-case complexity \citep{ProblemComplexityMethod1983nemirovskij}.

\begin{theorem}[First-Order Nuclear-Stationarity Lower-Bound] \label{thm:lower_bound}
    Let $\Dz, L, \sigma > 0$, $p \in (1, 2]$ and assume $\eps < \min\set{\frac \sigma 4, \sqrt{\Dz L} / 2}$. Then, for all $m, n \in \N_{\geq 2}$, we have
    \begin{equation*}
        \nucCompl \geq \frac{\min\set{m, n}} 2 \cdot \pare{\frac{\sigma}{4 \eps}}^{\frac p {p-1}}.
    \end{equation*}
\end{theorem}

We give an informal proof sketch, the complete proof can be found in \Cref{sec:app.lower_bounds}. Therefore let $(e_i)_{i=1}^{m}$ and $(f_i)_{i=1}^{n}$ be standard bases of $\R^m$ and $\R^n$ respectively, and denote $E_i \coloneqq e_i f_i^\top \in \R^{m \times n}$. Furthermore let $ q \coloneqq \pare{\frac{4 \eps}{\sigma}}^{\frac p {p-1}} \in (0, 1]$, $s \in \set{\pm 1}^{\kappa}$ be a sign vector and $S_s \coloneqq \sum_{i=1}^{\kappa} s_i E_i$, where $\kappa \coloneqq \min \set{m,n}$. Now we construct the family of hard instances as
\begin{align*}
    F_s(x) &\coloneqq \frac{L}{2\kappa} \fnorm{x}^2 - \frac{2 \eps}{\kappa} \langle S_s, x\rangle_F,\qquad
    \nabla f_s (x, \xi) \coloneqq
    \frac L {\kappa} x - \frac{2 \eps}{q} B s_I E_I,
\end{align*}
where $\xi = (B, I)$ with $B \sim \operatorname{Bern}(q)$ and $I \sim \operatorname{Unif}([\kappa])$ independently. Starting at zero, one can then roughly show that the algorithm must discover the sign of each $s_i$ to reach an $\eps$-$\nucnorm \cdot$-stationary point. However, in order to discover the sign of one $s_i$, $(B, I) = (1, i)$ must occur, which happens with probability $\frac q {\kappa}$. Hence, by a coupon collector type argument, we need at least $\frac {\kappa} q$ iterations in expectation to discover the sign of all $s_i$, which gives the desired lower-bound after showing that $F_s$ and $\nabla f_s$ satisfy the required smoothness and noise properties.

{
We note that the above lower-bound matches \Cref{thm:muon_convergence} in its dimension and $\sigma$ dependence, but not in its $\eps$ dependence. We will now present a lower-bound that addresses this gap, matching the exact dependence. It comes at the cost of a high-dimensional construction, which is known to be necessary to achieve the optimal $\eps$ dependence \citep{LowerBoundsNonConvex2023arjevani}. For simplicity, we only prove a version for zero-respecting algorithms, i.e., algorithms which satsify $\mathrm{supp}(\xtp) \subseteq \bigcup_{\tau = 1}^t \mathrm{supp}(g_\tau)$, the extension to randomised follows with the randomised rotation argument \citep{LowerBoundsNonConvex2023arjevani}.

\newcommand{\zrCompl}{\mathfrak m_{\eps}^{\text{zr}}\pare{\nucnorm \cdot, \Dz, L, p, \varsym}}
\begin{theorem}
    Let $c \coloneqq 1 / 300, \Dz, L, \sigma > 0$, $p \in (1, 2]$ and assume $0 < \eps \leq c \min\set{\sigma, \sqrt{\Dz L}}$. Then, for all $m, n \in \N_{\geq 2}$ with
    $
        \max\set{m, n} \geq \min\set{m, n} \frac{\Dz L}{\eps^2}
    $
    we have
    \begin{equation*}
        \zrCompl \geq c^2 \min\set{m, n} \cdot \frac{\Dz L}{\eps^2}\pare{\frac{c \varsym}{\eps}}^{\frac{p}{p-1}}.
    \end{equation*}
    In particular, \muon\ is first-order optimal for $\nucnorm \cdot$-stationarity.
\end{theorem}

The formal proof can be found in \Cref{sec:app.lower_bounds.high_dim} and is based on embedding $\kappa$ instances of the $\R^d$ construction \citep{LowerBoundsNonConvex2023arjevani} into matrices of size $\kappa d \times \kappa$. The dimension requirement follows from the fact that the $\R^d$ construction requires $d \gtrsim \frac{\Dz L}{\eps^2}$ to achieve the optimal $\eps$ dependence.
}

\section{Experiments}
\label{sec:experiments}

This section aims to empirically validate the theoretical findings of this paper. Since it was empirically observed that the heavy-tailed gradient distribution occur in language modelling \citep{WhyGradientClipping2020zhang,ahn2024linear}, we focus on this task.\footnote{Our code is available on Github: \url{https://github.com/fhueb/ht-schatten-experiments}}

\textbf{Experimental Details.} We pre-train the 70m parameter architecture of the \texttt{Pythia} series of models \citep{PythiaSuiteAnalyzing2023biderman} from scratch on the FineWeb-Edu dataset\footnote{\url{https://huggingface.co/datasets/HuggingFaceFW/fineweb-edu}} \citep{penedo2024fineweb} for Chinchilla-optimal 1.4B tokens \citep{EmpiricalAnalysisComputeOptimal2022hoffmann}. Following previous work, we choose the sequence length to be 2048 and the batch size to be 512 \citep{PythiaSuiteAnalyzing2023biderman}. Additionally we use the standard hyper-parameters $\beta = 0.95, \lambda = 0.1$ for \texttt{MUON} and use the WSD learning rate schedule \citep{zhai2022scaling,hu2024minicpm,hagele2024scaling} with a warmup of $5\%$ and a cooldown period of $20\%$. We tune learning rates across the grid $[0.005, 0.0075, 0.01, 0.0125, 0.015, 0.03, 0.06]$ and repeat the experiments across the seeds $0, 1, 2$. Further details can be found in \Cref{sec:app.experimental_details}.

\begin{figure}[t]
    \centering
    \begin{subfigure}[t]{0.49\textwidth}
        \centering
        \includegraphics[width=\textwidth]{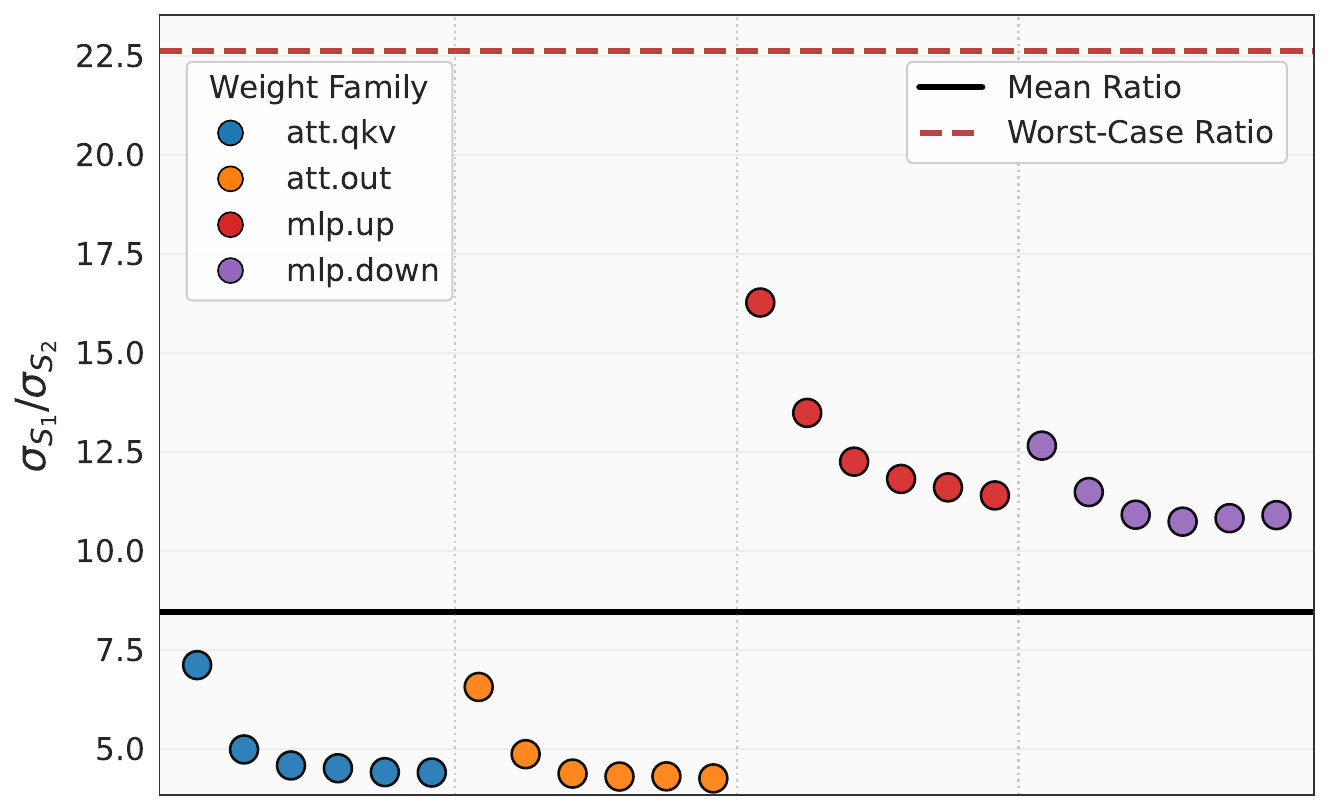}
        \caption{Ratios at Initialisation}
        \label{fig:noise_ratios_layers_init}
    \end{subfigure}
    \hfill
    \begin{subfigure}[t]{0.49\textwidth}
        \centering
        \includegraphics[width=\textwidth]{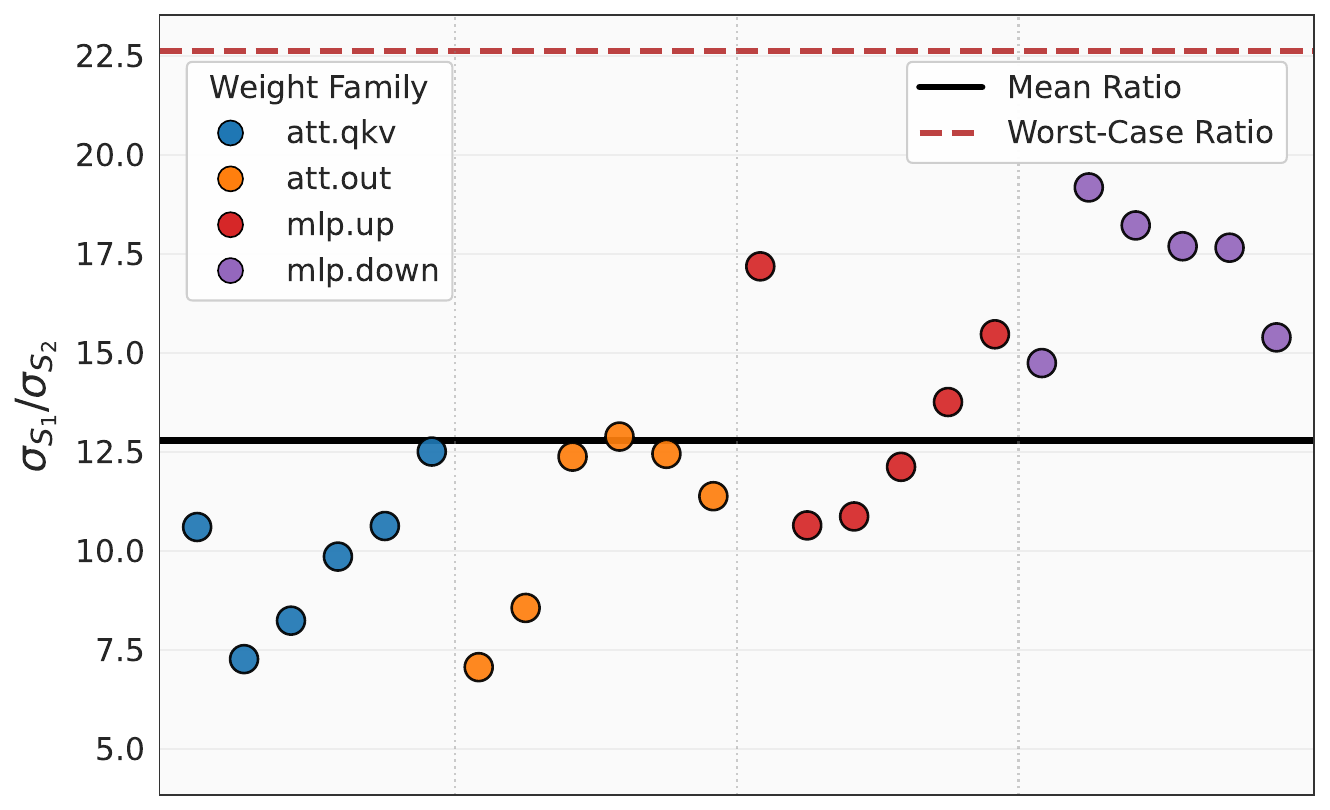}
        \caption{Ratios at Final Checkpoint}
        \label{fig:noise_ratios_layers_final}
    \end{subfigure}
    \caption{Ratios $\sigma_{S_1,p}/\sigma_{F,p}$ for $p = 1.5$ across all weight matrices trained by MUON. Colours indicate weight families, ordered by layer depth within. While non-trivial, the noise ratios stay considerably below the worst-case ratio $\sqrt{\min\set{m ,n}} \approx 22.5$.\vspace{-3mm}}
    \label{fig:noise_ratios_layers}
\end{figure}

\textbf{Norm Change in Noise Assumption.}
\Cref{sec:main} analyses \ourAlg\ directly in its native geometry rather than reducing to the Euclidean case. In particular, we consider the quantity
\begin{equation*}
    \sigma_{\dualnorm \cdot, p}^{p} \coloneqq \Exp{\dualnorm{\nf(x, \xi) - \nFx}^p} 
    \quad \text{rather than} \quad
    \sigma_{F, p}^{p} \coloneqq \Exp{\norm{\nf(x, \xi) - \nFx}_F^p}
\end{equation*}
which would be consistent with previous work. As noted by \citet{ConvergenceAnalysisMuon2025shen}, such a norm change in assumptions may implicitly introduce dimension dependencies. In the case of \muon, i.e., $\dualnorm \cdot = \norm \cdot_{S_1}$, we have the worst-case guarantee $\sigma_{S_1, p} \leq \sqrt{\min\set{m,n}} \sigma_{F, p}$. 
We therefore empirically calculate the ratio $\nicefrac{\sigma_{S_1, p}}{\sigma_{F, p}}$ for $p = 1.5$ for all weight matrices trained by \muon, as shown in \Cref{fig:noise_ratios_layers}. 
We observe that, while non-trivial, the noise ratios are far below the worst-case factor $\sqrt{\min\set{m, n}} \approx 22.5$. Furthermore the ratio is significantly lower at initialisation than at the final checkpoint, which may further explain the observation that \muon\ performs particularly well at the start of training \citep{MuonScalableLLM2025liu}. 
We repeat the same experiment while adding the input and output embedding layer, but measure the noise in $\ell$-geometry, reporting a $[0,1]$ normalised ratio in \Cref{fig:embedding_plots}. Matrices that were empirically observed to be more efficiently trained by \muon\ have a lower Schatten ratio, while we observe that embeddings have a lower $\ell$ ratio. This possibly explains why embedding layers are better trained by \textsc{Adam}.
For additional details we refer to \Cref{sec:app.experimental_details} and \citet{ConvergenceAnalysisMuon2025shen}, who conduct similar experiments for the smoothness constant.

\begingroup
\setlength{\intextsep}{1mm}
\begin{wraptable}{r}{0.27\textwidth}
\vspace{-1mm}
\centering
\small
\caption{
Final perplexity for Schatten-\(r\) geometries.
}
\label{tab:schatten_comparison}
\vspace{-0.5em}
\renewcommand{\arraystretch}{1.05}
\begin{tabular}{@{}cc@{}}
\toprule
\(r\) & Perplexity \(\downarrow\) \\
\midrule
\(1\)               & \(88.89 \pm 0.48\) \\
\(\nicefrac{8}{7}\) & \(86.50 \pm 0.91\) \\
\(\nicefrac{4}{3}\) & \(67.61 \pm 5.98\) \\
\(\nicefrac{3}{2}\) & \(60.57 \pm 4.48\) \\
\(2\)               & \(40.18 \pm 0.09\) \\
\(3\)               & \(34.51 \pm 0.13\) \\
\(4\)               & \(33.54 \pm 0.04\) \\
\(8\)               & \(\mathbf{33.02 \pm 0.01}\) \\
\(\infty\)          & \(\underline{33.34 \pm 0.02}\) \\
\bottomrule
\end{tabular}
\vspace{-1mm}
\end{wraptable}

\textbf{Impact of Norm Choice.}
To confirm the theoretical insights from \Cref{thm:main_conv}, we compare \ourAlg\ in different Schatten-$r$ geometries, $r \in \set{1, \nicefrac 8 7, \nicefrac 4 3, \nicefrac 3 2, 2, 3, 4, 8, \infty}$. Here $r = 2$ corresponds to the standard Euclidean geometry, while $r = \infty$ corresponds to the spectral geometry of \muon. Our theory suggests that increasing $r$ should initially improve performance: larger $r$ corresponds to a stronger dual stationarity measure until $r \approx \nicefrac{p}{p-1}$. Beyond this point, the stronger stationarity competes with additional dimension dependence it introduces (\Cref{thm:convergence_nssdm_norm_equivalence_formulation}), so the performance may plateau or decrease.

This is precisely the behaviour observed in \Cref{tab:schatten_comparison}. Performance improves sharply from $r=1$ to $r=3$, after which it essentially plateaus. These results are consistent with a tail index of approximately $p = 1.5$ for the gradient noise. We also observe that the spectral geometry of \muon, corresponding to \(r=\infty\), substantially outperforms the Euclidean geometry, in line with \Cref{thm:muon_convergence}. Finally, although intermediate Schatten geometries perform competitively, they require a full SVD, whereas \muon\ can use Newton--Schulz iterations. Hence \muon, which automatically adapts to the optimal geometry while being computationally efficient, is still preferable.

\begin{figure}[t]
    \centering
    \begin{subfigure}[t]{0.49\textwidth}
        \centering
        \includegraphics[width=\textwidth]{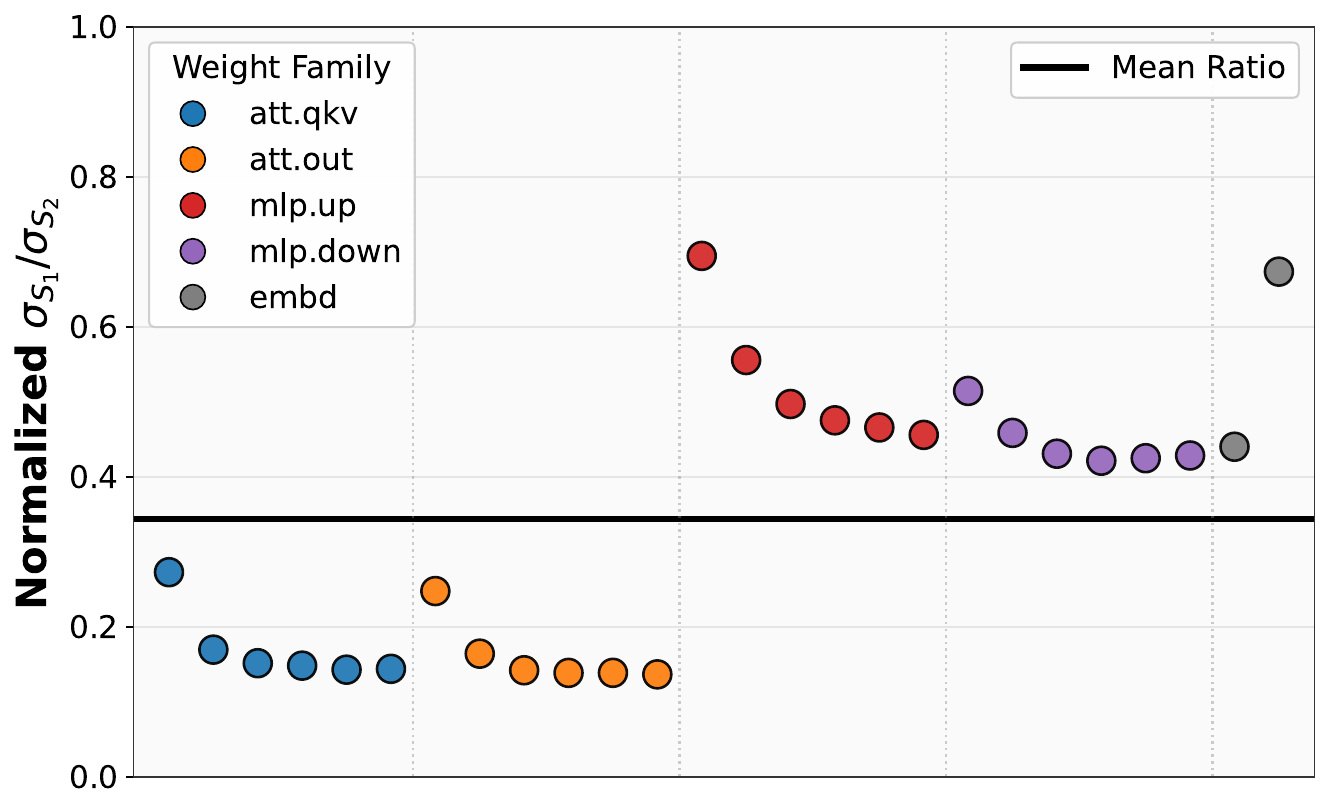}
        \caption{Ratios in $S_1$.}
        \label{fig:embedding_plots.Schatten}
    \end{subfigure}
    \hfill
    \begin{subfigure}[t]{0.49\textwidth}
        \centering
        \includegraphics[width=\textwidth]{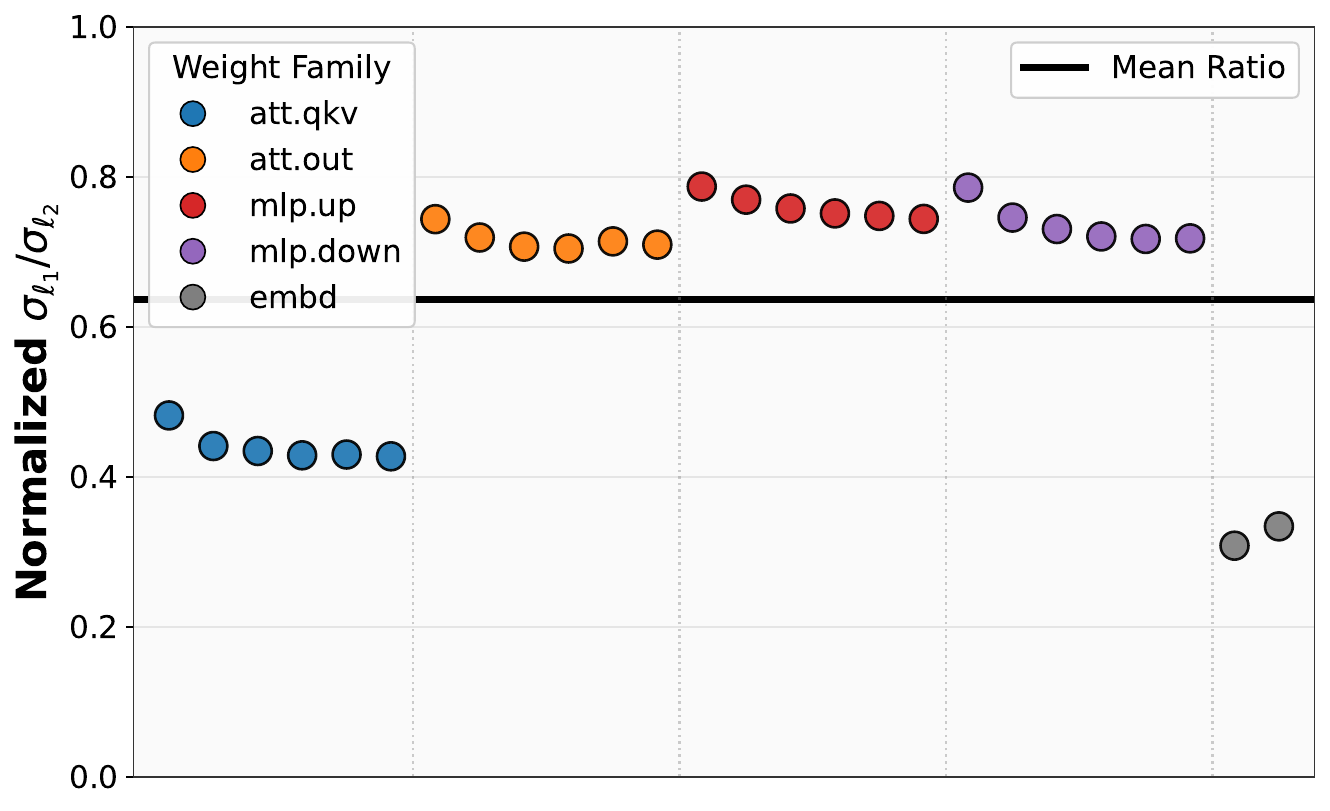}
        \caption{Ratios in $\ell_1$.}
        \label{fig:embedding_plots.ell}
    \end{subfigure}
    \caption{Normalised noise ratios including the embedding and unembedding layer at initialisation. A normalised value of $0$ corresponds to $\sigma_1 = \sigma_2$, a value of $1$ corresponds to $\sigma_1 = \rho\sigma_2$, where $\rho$ is the worst-case ratio. Interestingly, we find that the ratio is significantly smaller for the embedding and unembedding layers for $\sigma_{\ell_1}/\sigma_{\ell_2}$ possibly explaining why the $\ell_\infty$-norm is favourable over Schatten-$\infty$ norm in practice for those layers.\vspace{-5mm}}
    \label{fig:embedding_plots}
\end{figure}

\endgroup
\section{Related Work}
\label{sec:related}

\textbf{Steepest Descent in General Norms.} While classical \textsc{Sgd} can be seen as steepest descent in the Euclidean $\ell_2$ norm, early works have proposed to consider general-norm (normalised) steepest descent algorithms \citep[Section 9.4]{ConvexOptimization2004boyd}. In particular, \citet{AlmostLinearTimeAlgorithmApproximate2014kelner} proposed considering $\ell_\infty$-steepest descent for specific graph tasks, and \citet{StochasticSpectralDescent2015carlson} suggest considering $S_\infty$-steepest descent for Restricted Boltzmann Machines and Deep Learning \citep{IntroCD2015Carlson}.
Independently, \citet{OrthogonalisingGradientsSpeed2022tuddenham} proposed normalised $S_\infty$-steepest descent with momentum for machine learning tasks under the name \textsc{Orthogonal-Sgdm}. This algorithm first orthogonalises each stochastic gradient and then applies momentum to the resulting orthogonalised directions. Neither approach was widely adopted until recently, when \citet{MUON2024Bernstein} proposed using Newton-Schulz iterations  \citep{NewtonSchulz1970Kovarik} to efficiently compute the orthogonalisation step. The blog \citep{MuonOptimizerHidden2024jordan} popularised the normalised $S_\infty$-steepest descent approach under the name \muon, by changing the order of orthogonalisation and momentum in \textsc{Orthogonal-Sgdm}. 
Recently, \citet{SCION2025pethick} and \citet{MuonScalableLLM2025liu} suggest \muon\ with decoupled weight decay as a natural extension from the theoretical and empirical perspective respectively.

\textbf{Theory of \muon.}
Motivated by the strong empirical performance of \muon, a rapidly growing line of literature studies its convergence properties. \citet{NoteConvergenceMuon2025li} and \citet{ConvergenceAnalysisMuon2025shen} analyse \muon\ under $L$-smoothness and Euclidean bounded-variance, and later relate its behaviour to low-rank and approximately blockwise-diagonal Hessian structure.
Going beyond the $S_\infty$ geometry of \muon, \citet{SCION2025pethick,UnderstandingGradientOrthogonalization2025kovaleva} linked this family of algorithms to linear minimisation oracle (lmo) based algorithms, and establish convergence for arbitrary norms in finite-dimensional spaces.
Guarantees for \muon\ with decoupled weight decay \citep{SCION2025pethick} and extensions to Nesterov momentum \citep{MuonOptimizesSpectral2025chena,LionsMuonsOptimization2025sfyraki} draw a close connection to stochastic conditional gradient \citep{StochasticConditionalGradient2020mokhtari} and Lion-$\mathcal K$.
While the above works make their noise assumptions in the Euclidean space, \citet{NonEuclideanSGDStructured2025kovalev} consider the weighted-Euclidean assumption $\E[\norm{\nf(x, \xi) - \nFx}_{\Sigma^{-1}}^{\smash{\raisebox{-0.25ex}{$\scriptstyle 2$}}}] \leq \nucnorm{\Sigma}$, where $\Sigma$ is a self-adjoint positive definite operator and $\norm{v}_{\Sigma^{-1}}^{\smash{\raisebox{-0.25ex}{$\scriptstyle 2$}}} = {v^{\smash{\raisebox{-0.25ex}{$\scriptstyle \top$}}} \Sigma^{\smash{\raisebox{-0ex}{$\scriptstyle -1$}}} v}$.

\textbf{Heavy-Tailed Noise.} A separate line of work suggests that bounded variance may be too restrictive for modern deep learning, and instead suggests studying stochastic gradients with bounded $p$-th central moments for $p \in (1,2]$ \citep{WhyAreAdaptive2020zhang,ProblemComplexityMethod1983nemirovskij,ProximalPolicyOptimizations2021garg}. In Euclidean geometry, a large body of work establishes convergence of gradient-clipping based methods under heavy-tailed noise \citep[e.g.,][]{WhyAreAdaptive2020zhang,HighProbabilityBoundsStochastic2023sadiev,HighprobabilityBoundsNonConvex2021cutkosky}. Recent works establish optimal sample complexities for \ourAlg\ without the need for gradient-clipping under heavy-tailed noise \citep{GradientClippingNormalization2024hubler,NonconvexStochasticOptimization2025Liu}. 
Beyond Euclidean spaces, \citet{HighprobabilityBoundsNonConvex2021cutkosky} provide guarantees for $2$-uniformly convex Banach spaces, but require an additional gradient clipping step. This extension covers Schatten-$r$ norms for $r \in (1,2]$, but not the Schatten-$\infty$ geometry required for \muon. Complementary literature examines heavy-tailed \emph{data} distributions \citep{kunstner2026scaling,kim2026sharp}.

\textbf{Concurrent Work.} While preparing this paper, we became aware of the concurrent preprint \citep{ConvergenceMuonNonconvex2026zhang}. The authors analyse \muon\ under heavy-tailed noise and ($L_0, L_1$)-smoothness, where both assumptions are stated in the Frobenius norm. This results in a $\Oc\pare{\rho^{\frac{3p-2}{p-1}} \frac{\Dz L}{\eps^2} \pare{\frac{\euclvar}{\eps}}^{\frac{p}{p-1}}}$ convergence guarantee, which does not improve over the stationarity adjusted guarantee of Euclidean algorithms (see \Cref{eq:euclidean_stationarity_adjusted_guarantee}).
In contrast, our work considers general norms beyond $S_\infty$, assumptions stated in the native optimisation geometry and obtains a strictly improved dimension dependence. Additionally we provide lower-bounds showing that our dimension dependence is first-order optimal.

\section{Conclusion and Limitations}
\label{sec:conclusion}

We provide convergence guarantees for normalised steepest descent with momentum in general Banach spaces under heavy-tailed noise. 
In particular, for \muon, we obtain a nuclear-stationarity guarantee with substantially improved dimension dependence compared to the corresponding stationarity-adjusted Euclidean guarantees. We further prove that the remaining dimension dependence is first-order optimal and therefore cannot be removed.
Our experiments connect these results to empirical observations in language-model training: varying the Schatten-$r$ descent geometry produces the predicted improvement from Euclidean geometry to larger $r$, and weight matrices trained by MUON exhibit favourable Schatten noise ratios, while embeddings appear better aligned with an $\ell_\infty$-type geometry.

Our work has several limitations. The empirical evaluation is restricted to a 70M-parameter model trained on a single dataset, so the observed phenomena should be validated across larger variants. On the theoretical side, our guarantees are in expectation, and extending them to high-probability bounds under heavy-tailed noise remains open. Finally, the analysis studies idealised normalised steepest-descent updates, whereas practical \muon\ implementations rely on approximate Newton–Schulz iterations.

\acksection
FH gratefully acknowledges financial support from the ETH research grant and Swiss National Science Foundation (SNSF) Project Funding No.~200021-207343, the Alexander von Humboldt Foundation, and the European Union’s Horizon Europe research and innovation programme under grant agreements No.~101120237 (ELIAS) and No.~101070617 (ELSA). Views and opinions expressed are however those of the author only and do not necessarily reflect those of the European Union or European Commission. Neither the European Union nor the granting authority can be held responsible for them.
SS gratefully acknowledges generous support from the Alexander von Humboldt Foundation.

\printbibliography

\clearpage
\appendix
\tableofcontents
\clearpage
\section{Technical Lemmas}
\label{sec:app.technical}

\begin{lemma}[{\cite[see][Equation (167)]{NonlinearSpectralCalculus2014mendel}}]\label{lem:make_smoothness_weaker}
    Let $(X, \norm \cdot)$ be a Banach space.
    \begin{enumerate}[label=\roman*)]
        \item For all $1 < p' \leq p \leq 2$ we have $S_{p'}(X) \leq S_p(X)$.
        \item For all $2 \leq q \leq q' < \infty$ we have $K_{q'}(X) \leq K_q(X)$.
    \end{enumerate}
\end{lemma}

\begin{lemma}[see \cite{matrixAnalysis2013bhatia}]
\label{lem:schatten_norm_equivalence_constants}
    Let $1 \leq p \leq q \leq \infty$. Then the Schatten-$p$ and Schatten-$q$ norms satisfy 
    \begin{equation*}
        \norm{A}_{S_q} \leq \norm{A}_{S_p} \leq \rho \norm{A}_{S_q},
    \end{equation*}
    where $\rho = \min\set{m, n}^{\frac 1 p - \frac 1 q}$, for all $A \in \R^{m \times n}$.
\end{lemma}

\section{\texorpdfstring{$(L_0, L_1)$}{(L0, L1)}-Smoothness}
\label{sec:app.lzlo}
We will derive our results for the weaker smoothness notion of $(L_0, L_1)$-smoothness. This section contains the related work, and required definitions and properties of this smoothness notion in general Banach spaces.

\subsection{Related Work on \texorpdfstring{$(L_0, L_1)$}{(L0, L1)}-Smoothness}
\label{sec:app.lzlo.related_work}

Motivated by observations in language modelling, \citet{WhyGradientClipping2020zhang} propose considering $(L_0,L_1)$-smoothness, allowing the smoothness to grow affinely with the gradient norm. A long line of works analyse diverse algorithms in this setting \citep[e.g.,][]{RobustnessUnboundedSmoothness2022crawshaw,UniformSmoothnessStopped2023faw,ConvergenceAdamRelaxed2023li,chezhegov2025convergence}. In Euclidean geometry, \citet{ImprovedAnalysisClipping2020zhang,ConvergenceImprovementStochastic2021zhao} prove the convergence of \ourAlg, and \citet{NonconvexStochasticOptimization2025Liu} extend this result to the heavy-tailed setting. Beyond Euclidean geometry, \citet{GeometrySignGradient2020balles} provide guarantees for the $\ell_\infty$-geometry in the deterministic setting, and \citet{GeneralizedGradientNorm2025pethick} provide convergence guarantees in arbitrary finite-dimensional spaces in the stochastic setting by adding gradient-clipping, while \citet{GluonMakingMuon2025riabinin}\footnote{We note that this preprint has a typo. While their variance assumption (Assumption 2) is stated in the dual norm, the authors use the inner product structure of the Frobenius norm in the proof of Theorem 2. In particular, a fix requires switching to $\euclvar$ and introduces a $\rho^2$ dependence.} concurrently prove convergence without the clipping step.

\subsection{Properties of \texorpdfstring{$(L_0, L_1)$}{(L0, L1)}-Smoothness in Banach Spaces}
\label{sec:app.lzlo.properties}

\begin{definition}[$(L_0, L_1)$-smoothness]\label{def:lzlo}
    Let $L_0, L_1 \geq 0$. A function $F \colon \Bc \to \mathbb{R}$ is said to be \emph{$(L_0, L_1)$-smooth} if it is differentiable and, for all $x,y \in \Bc$,
    \begin{equation*}
        \dualnorm{\nabla F(x) - \nabla F(y)} \leq (L_0 + L_1 \dualnorm{\nFx}) A(L_1 \norm{x-y}) \norm{x-y},
    \end{equation*}
    where $A(c) \coloneqq \frac{e^c - 1}c$.
\end{definition}

\begin{lemma}[Smoothness Upper-Bound]\label{lem:lzlo_upper_bound}
    Let $F \colon \Bc \to \mathbb{R}$ be an $(L_0, L_1)$-smooth function. Then, for all $x,y \in \Bc$,
    \begin{equation*}
        F(y) \leq F(x) + \langle \nabla F(x), y - x \rangle + \pare{L_0 + L_1 \dualnorm{\nFx}} B(L_1 \norm{x-y}) \norm{x-y}^2,
    \end{equation*}
    where $B(c) = \frac{e^c - 1 - c}{c^2}$.
\end{lemma}
\begin{proof}
    This directly follows with the same proof as \citep[Lemma 2.5]{OptimizingL0L1SmoothFunctions2024vankov} by noting that all used inequalities also hold in general Banach spaces.
\end{proof}

\section{Proofs of Main Results}
\label{sec:app.proofs}

\subsection{Proofs of \texorpdfstring{\Cref{sec:main.general_norms}}{Section 3.1}}
\label{sec:app.upper_bounds.general}

This section contains generalisations of the main results from \Cref{sec:main} and their proofs. We will derive all results for the weaker $(L_0, L_1)$-smoothness notion popularised by \citet{WhyGradientClipping2020zhang}.

\begin{assum}[$(L_0, L_1)$-smoothness]\label{assum:lzlosmooth}
    We assume the objective function $F$ is $(L_0, L_1)$-smooth.
\end{assum}
\newcommand{\lzlosmooth}{\nameeqref{assum:lzlosmooth}}

We first derive a slightly improved descent lemma over \citep{GeneralizedGradientNorm2025pethick}.

\begin{lemma}[Descent Lemma]\label{lem:descent_lemma}
    Suppose \Cref{assum:lower,assum:lzlosmooth} hold. Furthermore let $T \in \Ngeq$ and consider the iterates $(\xt)$ generated by \ourAlg. Then
    \begin{equation*}
        \iSum \pare{\sst - L_1 \sst^2 B_t} \nnFxt_* 
        \leq 
        \Delta_1
        + L_0 \iSum \sst^2 B_t
        + 2 \iSum \sst \norm{\nFxt - \mt}_*,
    \end{equation*}
    where $B_t \coloneqq B(L_1 \sst)$, $B(c) = \frac{e^c - 1 - c}{c^2}$.
\end{lemma}

\begin{proof} This proof generalises \citep[Lemma 12]{ParameterAgnosticOptimizationRelaxed2023hubler} to Banach spaces. Using \Cref{lem:lzlo_upper_bound} we get
    \begin{align*}
        F(\xtp) - F(\xt)
        &\leq \langle \nFxt, \xtp - \xt \rangle + \pare{L_0 + L_1 \norm{\nFxt}_*} B(L_1 \sst) \norm{\xtp - \xt}^2 \\
        &= \sst \langle \nFxt, \lmo(\mt) \rangle + \pare{L_0 + L_1 \norm{\nFxt}_*} B(L_1 \sst) \sst^2.
    \end{align*}
    Furthermore
    \begin{align*}
        \langle \nFxt, \lmo(\mt) \rangle
        &= \langle \nFxt - \mt, \lmo(\mt) \rangle + \langle \mt, \lmo(\mt) \rangle \\
        &\leq \dualnorm{\nFxt - \mt} \norm{\lmo(\mt)} - \dualnorm{\mt}\\
        &\leq 2 \dualnorm{\nFxt - \mt} - \dualnorm{\nFxt},
    \end{align*}
    where we used the definition of $\lmo$\ and the dual norm in the first, and the triangle inequality in the second step. Combining yields
    \begin{align*}
        F(\xtp) - F(\xt)
        &\leq \sst \left(2 \dualnorm{\nFxt - \mt} - \dualnorm{\nFxt}\right) + \pare{L_0 + L_1 \norm{\nFxt}_*} B(L_1 \sst) \sst^2\\
        &= \pare{L_1 \sst^2 B(L_1 \sst) - \sst}\dualnorm{\nFxt} + 2\sst \dualnorm{\nFxt - \mt} + L_0 B(L_1 \sst) \sst^2.
    \end{align*}
    Summing up and re-arranging finishes the proof.
\end{proof}

Next we derive the main technical lemma of our analysis, which bounds the deviation of the momentum from the true gradient for general norms, without requiring a conversion to the Euclidean norm.

\begin{lemma}\label{lem:mom:deviation_bound}
    Suppose \Cref{assum:noise,assum:lzlosmooth} hold and that $(\Bs, \dualnorm \cdot)$ is $p'$-uniformly smooth with smoothness constant $S_{p'}(\Bs) \leq S$. Then the iterates generated by \ourAlg\ satisfy
    \begin{equation*}
        \Exp{\dualnorm{\nFxt - \mt}} 
        \leq 
        \varsym S \pare{\sum_{\tau=1}^t b_{\tau,\ii}^r}^{1/r}
        + L_0 \sum_{\tau=2}^t c_{\tau,\ii} 
        + L_1 \sum_{\tau=2}^t c_{\tau,\ii} \Exp{\norm{\nF \pare{x_\tau}}_*},
    \end{equation*}
    where $r \coloneqq \min \set{p, p'}$ and, for easier readability, we define 
    \begin{align*}
        b_{1,\ii} \coloneqq \prod_{\kappa = 2}^\ii \iterationmom \kappa, \qquad \text{and for $\tau \geq 2$,} \qquad
        b_{\tau,\ii} &\coloneqq (1-\iterationmom \tau) \prod_{\kappa = \tau + 1}^\ii \iterationmom \kappa, \qquad 
        c_{\tau,\ii} \coloneqq \ssi_{\tau-1} A_{\tau - 1} \prod_{\kappa = \tau}^\ii \iterationmom \kappa,
    \end{align*}
    where $A_\tau = A(L_1 \iterationss \tau) = \frac{e^{L_1 \iterationss \tau} - 1}{L_1 \iterationss \tau}$. 
    In particular, for constant parameters $\sst \equiv \ssi$ and $\momt \equiv \beta$, we get
    \begin{equation*}
        \iSum \Exp{\dualnorm{\nFxt - \mt}}
        \leq
         \frac{\varsym S}{1-\beta}
        + \varsym S T (1-\beta)^{\frac{r-1}{r}}
        + L_0 \frac{\ssi A T}{1-\beta}
        + L_1 \frac{\ssi A}{1-\beta} \iSum \Exp{\norm{\nF \pare{\xt}}_*},
    \end{equation*}
    where $A \equiv A(L_1 \ssi)$.
\end{lemma}
\begin{proof}
    First note that the standard decomposition of $\mu_t \coloneqq \mt - \nFxt$ into a noise and smoothness term \citep{MomentumImprovesNormalized2020cutkosky},
    \begin{equation*}
        \mu_t = 
        \sum_{\tau = 1}^t b_{\tau, \ii} \pare{\nf \pare{x_\tau, \xi_\tau} - \nF \pare{x_\tau}} 
        + \sum_{\tau = 2}^t \prod_{\kappa = \tau}^t \iterationmom \kappa \pare{\nF\pare{x_{\tau-1}}-\nF \pare{x_\tau} },
    \end{equation*}
    is norm independent and hence also holds in general Banach spaces. Next, we have
    \begin{equation*}
        \dualnorm{\nF\pare{x_\tau} - \nF\pare{x_{\tau-1}}}
        \leq \pare{L_0 + L_1 \norm{\nF \pare{x_\tau}}_*} A(L_1 \iterationss {\tau-1}) \norm{x_\tau - x_{\tau - 1}}
        \leq L_\tau A_{\tau-1} \iterationss{\tau - 1},
    \end{equation*}
    where we use the slight abuse of notation $L_\tau \coloneqq L_0 + L_1 \norm{\nF \pare{x_\tau}}_*$. Hence, by our definition of $b_{\tau, \ii}$ and $c_{\tau, \ii}$ in the statement, the triangle inequality yields
    \begin{align*}
        \Exp{\dualnorm{\mu_t}}
        &\leq
        \Exp{\dualnorm{\sum_{\tau=1}^t b_{\tau,\ii} \pare{\nf\pare{x_\tau,\xi_\tau}-\nF\pare{x_\tau}}}}
        + \sum_{\tau=2}^t \pare{\prod_{\kappa = \tau}^\ii \iterationmom \kappa} \Exp{\dualnorm{\nF\pare{x_{\tau-1}}-\nF\pare{x_\tau}}} \\
        &\leq
        \Exp{\dualnorm{\sum_{\tau=1}^t b_{\tau,\ii} \pare{\nf\pare{x_\tau,\xi_\tau}-\nF\pare{x_\tau}}}}
        + \sum_{\tau=2}^t c_{\tau,\ii} \Exp{L_\tau}.
    \end{align*}
    Now note that $(\nf\pare{x_\tau,\xi_\tau}-\nF\pare{x_\tau})_{\tau = 1}^t$ is a martingale difference sequence in $\Bs$, and, by \Cref{lem:make_smoothness_weaker}, $\Bs$ is also $r$-uniformly smooth with $S_r(\Bs) = S_{p'}(\Bs) \leq S$. Hence, we may apply \Cref{thm:banach_martingale} to get
    \begin{align*}
        \Exp{\dualnorm{\sum_{\tau=1}^t b_{\tau,\ii} \pare{\nf\pare{x_\tau,\xi_\tau}-\nF\pare{x_\tau}}}}
        &\leq
        \Exp{\dualnorm{\sum_{\tau=1}^t b_{\tau,\ii} \pare{\nf\pare{x_\tau,\xi_\tau}-\nF\pare{x_\tau}}}^r}^{\nicefrac 1 r} \\
        &\leq S
        \pare{\sum_{\tau=1}^t b_{\tau,\ii}^r
        \Exp{\dualnorm{\nf\pare{x_\tau,\xi_\tau}-\nF\pare{x_\tau}}^r}}^{1/r} \\
        &\leq \varsym S
        \pare{\sum_{\tau=1}^t b_{\tau,\ii}^r}^{1/r},
    \end{align*}
    where we used $r \leq p$ and \pBCM\ in the last step. For the constant parameter version note that
    \begin{align*}
        \sum_{\tau = 1}^t b_{\tau,\ii}^r
        &= 
        \beta^{r(t-1)} +
        (1-\beta)^r \sum_{\tau = 2}^t \beta^{r(t-\tau)}
        \leq \beta^{r(t-1)} + \frac{(1-\beta)^r}{1-\beta^r} 
        \leq \beta^{r(t-1)} + (1-\beta)^{r-1}\\
        \sum_{\tau = 2}^t c_{\tau,\ii}
        &= \ssi A \sum_{\tau = 1}^{t-1} \beta^\tau \leq \frac{\ssi A}{1-\beta},
    \end{align*}
    where $A \equiv A(L_1 \ssi)$, and hence
    \begin{align*}
        \varsym S \iSum \pare{\sum_{\tau=1}^t b_{\tau,\ii}^r}^{1/r}
        & \leq \varsym S\pare{\frac{1}{1-\beta} + T (1-\beta)^{\frac{r-1}{r}}} \\
        L_0 \iSum \sum_{\tau=2}^t c_{\tau,\ii}
        &\leq L_0 \frac{\ssi A T}{1-\beta} \\
        L_1 \iSum \sum_{\tau=2}^t c_{\tau,\ii} \Exp{\norm{\nF \pare{x_\tau}}_*}
        &= L_1 \sum_{\tau = 2}^\li \pare{\sum_{\ii = \tau}^\li c_{\tau, \ii}} \Exp{\dualnorm{\nF \pare{\iterationx \tau}}}\\
        &\leq L_1 \sum_{\tau = 2}^\li \frac{\ssi A}{1-\momentum}\Exp{\dualnorm{\nF \pare{\iterationx \tau}}}\\
        &\leq \frac{L_1 \ssi A}{1-\beta} \iSum \Exp{\norm{\nF \pare{\xt}}_*},
    \end{align*}
    Combining the three bounds finishes the proof.
\end{proof}

Finally we are able to derive the main convergence result for \ourAlg\ in general Banach spaces.

\begin{theorem}\label{thm:main_lzlo}
    Suppose \Cref{assum:lower,assum:lzlosmooth,assum:noise} hold and that $(\Bc, \norm \cdot)$ is $q$-uniformly convex with constant $K_q(\Bc) \leq K$. Furthermore let $p'$ be the conjugate exponent of $q$, i.e., $\nicefrac{1}{q} + \nicefrac{1}{p'} = 1$, and $r = \min\set{p, p'}$. Then the iterates generated by \ourAlg\ with 
    \begin{align*}
        \sst &\equiv \ssi = \min\set{\frac{1-\momentum}{5L_1}, \sqrt{\frac{\Delta_1(1-\momentum)}{L_0 T}}}, \\
        \momt &\equiv \momentum = 1 - \min\set{1, \max \set{
            \pare{\frac{\Dz L_0}{\varsym^2 K^2 T}}^{\frac{r}{3r-2}},
            \pare{\frac{5L_1\Delta_1}{2\varsym K T}}^{\frac{r}{2r-1}},
            \li^{-\frac{r}{2r-1}}
        }}
    \end{align*}
    satisfy
    \begin{align*}
        \frac 1 \li \iSum \Exp{\nnFxt_*} 
        \leq &\ 
        12 \frac{\Dz L_1}{\li} 
        + 9 \sqrt{\frac{\Dz L_0} \li} 
        + 14 \pare{\frac{\Dz L_1\pare{\sigma K}^{\frac r {r-1}}}{\li}}^{\frac{r-1}{2r-1}}\\
        &\ + 14\pare{\frac{\Dz L_0 (\varsym K)^{\frac r {(r-1)}}}{\li}}^{\frac{r-1}{3r-2}}
        + \frac{10 \varsym K}{\li^{\frac {r-1} {2r-1}}}.
    \end{align*}
    In particular, the sample complexity to reach an $\eps$-$\dualnorm \cdot$-stationary point is at most
    \begin{equation*}
        \Oc \pare{
            \frac{\Dz L_1}{\eps} 
            + \frac{\Dz L_0}{\eps^2} 
            + \frac{\Dz L_1}{\eps} \pare{\frac{\sigma K}{\eps}}^{\frac r {r-1}}
            + \frac{\Dz L_0}{\eps^2} \pare{\frac{\sigma K}{\eps}}^{\frac r {r-1}}
            + \pare{\frac{\varsym K}{\eps}}^{\frac{2r-1}{r-1}}
        }.
    \end{equation*}
\end{theorem}

\begin{proof}
    Applying \Cref{lem:descent_lemma} yields
    \begin{align*}
        \ssi \iSum \Exp{\nnFxt_*}
        \leq&\  
        \Delta_1
        + L_0 \ssi^2 B \li
        + L_1 \ssi^2 B \iSum \Exp{\nnFxt_*}
        + 2 \iSum \sst \Exp{\norm{\nFxt - \mt}_*} \\
        \leq&\ 
        \Delta_1
        + L_0 \ssi^2 T \pare{B + \frac{2A}{1 - \momentum}}
        + 2 \varsym \ssi K \pare{\frac{1}{1-\momentum} + T (1-\momentum)^{\frac{r-1}{r}}}\\
        &\ 
        + L_1 \ssi^2 \pare{B + \frac{2A}{1-\momentum}} \iSum \Exp{\norm{\nF \pare{\xt}}_*}.
    \end{align*}
    Next we calculate 
    \begin{equation}\label{eq:technical_A_B_term_bound}
        L_1 \ssi^2 \pare{B + \frac{2A}{1-\momentum}}
        \leq \ssi \pare{\frac{L_1 \ssi }{1-\momentum} \frac{A(5 - \momentum)}{2}}
        {\leq \frac 5 9 \ssi}
    \end{equation}
    where we used $B \leq \nicefrac A 2$ in the first, and our choice of $\ssi$, which implies $\frac{L_1 \ssi}{1 - \beta} \leq \frac 1 5$ and $A \leq \nicefrac{10}9$, in the last inequality. Hence, re-arranging further yields
    \begin{align}\label{eq:plugged_in_eta}\begin{split}
        \frac 1 \li \iSum \Exp{\nnFxt_*}
        &\leq
        \frac{9\Delta_1}{4\ssi \li}
        + \frac{6 L_0 \ssi A}{(1 - \momentum)}
        + \frac{5 \varsym K }{T(1-\momentum)}
        + 5 \varsym K (1-\momentum)^{\frac{r-1}{r}}\\
        &\leq 9 \sqrt{\frac{\Dz L_0}{\li(1-\momentum)}} 
        + \frac{12 \Dz L_1}{T (1-\momentum)}
        + \frac{5 \varsym K }{T(1-\momentum)}
        + 5 \varsym K (1-\momentum)^{\frac{r-1}{r}}
    \end{split}\end{align}
    where we used that $1 / \min \set{a, b} \leq 1 / a + 1 / b$ and the definition of $\ssi$ in the last step. Next we upper bound the remaining $4$ terms by our choice of momentum. To simplify exposition, let us denote
    \begin{equation*}
        u \coloneqq \pare{\frac{\Dz L_0}{\varsym^2 K^2 T}}^{\frac{r}{3r-2}}, \qquad 
        v \coloneqq \pare{\frac{5L_1\Delta_1}{2\varsym K T}}^{\frac{r}{2r-1}}, \qquad 
        w \coloneqq \li^{-\frac r {2r-1}}
    \end{equation*}
    such that $1-\momentum = \min\set{1, \max\set{u,v,w}}$. First note that, due to $\frac 1 {\sqrt{1-\momentum}} \leq \frac 1 {\min \set{1, \sqrt{u}}} \leq 1 + \frac 1 {\sqrt u}$,
    \begin{equation*}
        9\sqrt{\frac{\Dz L_0}{T(1-\momentum)}} 
        \leq 9 \sqrt{\frac{\Dz L_0} \li} + 
        9\pare{\frac{\Dz L_0 (\varsym K)^{\frac r {(r-1)}}}{\li}}^{\frac{r-1}{3r-2}}.
    \end{equation*}
    For the second term we use the same argument as before, this time using $v$, to get
    \begin{equation*}
        \frac{12\Dz L_1}{T (1-\momentum)}
        \leq 12\frac{\Dz L_1}{\li} + 12 \pare{\frac 2 5}^{\frac r {2r-1}} \pare{\frac{ \Dz L_1\pare{\sigma K}^{\frac r {r-1}}}{\li}}^{\frac{r-1}{2r-1}}
        \leq 12\frac{\Dz L_1}{\li} + 7 \pare{\frac{ \Dz L_1\pare{\sigma K}^{\frac r {r-1}}}{\li}}^{\frac{r-1}{2r-1}}
    \end{equation*}
    For the third term note that $\frac 1 {1-\momentum} \leq \frac 1 w$ and hence
    \begin{equation*}
        \frac{5 \varsym K }{T(1-\momentum)}
        \leq \frac{5 \varsym K}{\li^{\frac {r-1} {2r-1}}}.
    \end{equation*}
    Finally, for the last term we use $(1-\momentum)^\alpha \leq u^\alpha + v^\alpha + w^\alpha$ to get
    \begin{align*}
        5 \varsym K (1-\momentum)^{\frac{r-1}{r}}
        &\leq 5 \varsym K \pare{\frac{\Dz L_0}{\varsym^2 K^2 \li}}^{\frac{r-1}{3r-2}}
        + 5 \varsym K \pare{\frac{5L_1\Delta_1}{2\varsym K \li}}^{\frac{r-1}{2r-1}}
        + 5 \varsym K \li^{-\frac{r-1}{2r-1}}\\
        &= 5 \pare{\frac{\Dz L_0 \pare{\varsym K}^{\frac r {r-1}}}{\li}}^{\frac{r-1}{3r-2}}
        + 5 \pare{\frac{5\Dz L_1 \pare{\varsym K}^{\frac r {r-1}}}{2\li}}^{\frac{r-1}{2r-1}}
        + 5 \frac{\varsym K}{\li^{\frac {r-1} {2r-1}}}
    \end{align*}
    Plugging these bounds into \eqref{eq:plugged_in_eta} and using $5 \pare{\frac 5 2}^{\frac{r-1}{2r-1}} \leq 7$ finishes the proof.
\end{proof}

\begin{remark}
    The above result can readily be extended to the parameter-agnostic setting by using the same arguments as in \citep{ParameterAgnosticOptimizationRelaxed2023hubler} to handle the non-constant parameters and error accumulation before $\sst \lesssim \nicefrac 1 {L_1}$.
\end{remark}

\subsection{Proofs of \texorpdfstring{\Cref{sec:main.muon}}{Section 3.2}}
\label{sec:app.proofs.muon}

In \Cref{thm:main_lzlo} we provide guarantees for arbitrary $p'$-uniformly smooth dual norms $\dualnorm \cdot$, but do not achieve optimal rates whenever $p' < p$. In this section we will use an orthogonal approach, where we instead assume that $\dualnorm \cdot \leq \rho \norm{\cdot}_p$ for some $p$-uniformly smooth norm $\norm \cdot _p$. Such a constant always exists in the finite dimensional case, and will allow us to achieve optimal rates. We first derive a general result for \ourAlg\ in \Cref{sec:app.proofs.muon.general_norm}, before specialising it to \muon\ in \Cref{sec:app.proofs.muon.muon}.

\subsubsection{Banach Space Setting}
\label{sec:app.proofs.muon.general_norm}

Therefore we first formulate \Cref{lem:mom:deviation_bound} for this approach.

\begin{lemma}\label{lem:mom:deviation_bound_norm_equivalence_formulation}
    Suppose \Cref{assum:noise,assum:lzlosmooth} hold, and that $\norm \cdot_p \leq \dualnorm \cdot \leq \rho \norm \cdot_p$, where $(\Bs, \norm \cdot_p)$ is $p$-uniformly smooth with smoothness constant $S_{p}(\Bs, \norm \cdot_p) \leq S$. Then the iterates generated by \ourAlg\ with constant parameters $\sst \equiv \ssi$ and $\momt \equiv \beta$ satisfy
    \begin{equation*}
        \iSum \Exp{\dualnorm{\nFxt - \mt}}
        \leq
        \frac{\rho \varsym S }{1-\beta}
        + \rho \varsym S T (1-\beta)^{\frac{p-1}{p}}
        + L_0 \frac{\ssi A T}{1-\beta}
        + L_1 \frac{\ssi A}{1-\beta} \iSum \Exp{\norm{\nF \pare{\xt}}_*},
    \end{equation*}
    where $A \equiv A(L_1 \ssi) = \frac{e^{L_1 \ssi} - 1}{L_1 \ssi}$. 
\end{lemma}

We note that such $\rho$ always exists in finite dimensional, but may not exist in general Banach spaces. 

\begin{proof}
    We proceed exactly as in the proof of \Cref{lem:mom:deviation_bound}, but replace
    \begin{align*}
        \Exp{\dualnorm{\sum_{\tau=1}^t b_{\tau,\ii} \pare{\nf\pare{x_\tau,\xi_\tau}-\nF\pare{x_\tau}}}}
        &\leq
        \Exp{\dualnorm{\sum_{\tau=1}^t b_{\tau,\ii} \pare{\nf\pare{x_\tau,\xi_\tau}-\nF\pare{x_\tau}}}^r}^{\nicefrac 1 r}
    \end{align*}
    with
    \begin{align*}
        \Exp{\dualnorm{\sum_{\tau=1}^t b_{\tau,\ii} \pare{\nf\pare{x_\tau,\xi_\tau}-\nF\pare{x_\tau}}}}
        &\leq \rho \Exp{\norm{\sum_{\tau=1}^t b_{\tau,\ii} \pare{\nf\pare{x_\tau,\xi_\tau}-\nF\pare{x_\tau}}}_p}\\
        &\leq \rho \Exp{\norm{\sum_{\tau=1}^t b_{\tau,\ii} \pare{\nf\pare{x_\tau,\xi_\tau}-\nF\pare{x_\tau}}}_p^p}^{\nicefrac 1 p}.
    \end{align*}
    The proof closes the same way as before by using the $p$-uniform smoothness of $\norm \cdot_p$.
\end{proof}

Now we are ready to prove the norm-constant version of the convergence result for \ourAlg\ under heavy-tailed noise and $(L_0, L_1)$-smoothness.

\begin{theorem}\label{thm:convergence_nssdm_norm_equivalence_formulation}
    Suppose \Cref{assum:lower,assum:noise,assum:lzlosmooth} hold. Further assume that $\norm \cdot _p \leq \dualnorm \cdot \leq \rho \norm \cdot_p$, where $(\Bs, \norm \cdot_p)$ is $p$-uniformly smooth with smoothness constant $S_{p}(\Bs, \norm \cdot_p) \leq S$. Then the iterates generated by \ourAlg\ with 
    \begin{align*}
        \sst &\equiv \ssi = \min\set{\frac{1-\momentum}{5L_1}, \sqrt{\frac{\Delta_1(1-\momentum)}{L_0 T}}},\\
        \momt &\equiv \momentum = 1 - \min\set{1, \max \set{
            \pare{{\frac{\Dz L_0}{\varsym^2 \rho^2 S^2 T}}}^{\frac{p}{3p-2}},
            \pare{\frac{5L_1\Delta_1}{2\varsym \rho S T}}^{\frac{p}{2p-1}},
            \li^{-\frac{p}{2p-1}}
        }}
    \end{align*}
    satisfy
    \begin{align*}
        \frac 1 \li \iSum \Exp{\dualnorm{\nFxt}} 
        \leq &\ 
        9 \sqrt{\frac{\Dz L_0} \li} 
        + 12 \frac{\Dz L_1} \li 
        + 14 \pare{\frac{\Dz L_0 \pare{\rho \varsym S}^{\frac p {p-1}}} \li}^{\frac{p-1}{3p-2}} \\
        &\ + 14 \pare{\frac{\Dz L_1 \pare{\rho \varsym S}^{\frac p {p-1}}} \li}^{\frac{p-1}{2p-1}} 
        + 10 \frac{\rho \varsym S}{\li^{\frac{p-1}{2p-1}}}.
    \end{align*}
    In particular, the sample complexity to reach an $\eps$-$\dualnorm \cdot$-stationary point is at most
    \begin{equation*}
        \Oc \pare{
            \frac{\Dz L_1}{\eps}
            + \frac{\Dz L_0}{\eps^2}
            + \frac{\Dz L_1}{\eps}\pare{\frac{\rho\varsym S}{\eps}}^{\frac p {p-1}}
            + \frac{\Dz L_0}{\eps^2}\pare{\frac{\rho \varsym S}{\eps}}^{\frac p {p-1}}
            + \pare{\frac{\rho \varsym S}{\eps}}^{\frac{2p-1}{p-1}}
        }.
    \end{equation*}
\end{theorem}

We note that \pBCM\ could be replaced with the $S_p$ version $\Exp{\norm{\nf\pare{x,\xi} - \nF\pare{x}}_{p}^p} \leq \sigma^p$ for all $x \in \Bc$, which would allow for smaller $\sigma$ values. However, we choose to state the result with the stronger \pBCM\ for better comparability with the previous results for \ourAlg.

\begin{proof}
    For notational conciseness denote $\kappa \coloneqq \rho \varsym S$. We proceed as in the proof of \Cref{thm:main_lzlo}, but replacing \Cref{lem:mom:deviation_bound} by \Cref{lem:mom:deviation_bound_norm_equivalence_formulation}, to get
    \begin{align*}
        \ssi \iSum \Exp{\nnFxt_*}
        \leq&\  
        \Delta_1
        + L_0 \ssi^2 \li \pare{B + \frac{2 A}{1-\momentum}}
        + \frac{2 \kappa \ssi }{1-\momentum} 
        + 2 \kappa \ssi \li (1-\momentum)^{\frac{p-1}p}\\
        &\ + L_1 \ssi^2 \pare{B + \frac{2A}{1-\momentum}} \iSum \Exp{\norm{\nF \pare{\xt}}_*}.
    \end{align*}
    By \Cref{eq:technical_A_B_term_bound} $L_1 \ssi^2 \pare{B + \frac{2A}{1-\momentum}} \leq \frac 5 9 \eta$ and re-arranging thus yields
    \begin{equation*}
        \frac 1 \li \iSum \Exp{\dualnorm{\nFxt}}
        \leq \frac{9\Dz}{4\ssi T}
        + \frac{6 L_0 \ssi A}{ (1 - \momentum)}
        + \frac{5 \kappa }{\li(1-\momentum)} 
        + 5 \kappa (1-\momentum)^{\frac{p-1}p}.
    \end{equation*}
    Using $\ssi = \min \set{\frac {1-\momentum} {5L_1}, \sqrt{\frac{\Delta_1(1-\momentum)}{L_0 T}}}$ we further obtain
    \begin{equation*}
        \frac 1 \li \iSum \Exp{\dualnorm{\nFxt}}
        \leq 
        9 \sqrt{\frac{\Dz L_0}{\li(1-\momentum)}} 
        + \frac{12 \Dz L_1}{\li (1-\momentum)}
        + \frac{5 \kappa }{\li(1-\momentum)} 
        + 5 \kappa (1-\momentum)^{\frac{p-1}p}.
    \end{equation*}
    Now denote 
    \begin{equation*}
        U \coloneqq \pare{\frac{\Dz L_0 {\kappa}^{\frac p {p-1}}} \li}^{\frac{p-1}{3p-2}}, \qquad
        V \coloneqq \pare{\frac{\Dz L_1 {\kappa}^{\frac p {p-1}}} \li}^{\frac{p-1}{2p-1}}, \qquad
        W \coloneqq \frac{\kappa}{\li^{\frac{p-1}{2p-1}}}
    \end{equation*}
    With the same arguments as in the proof of \Cref{thm:main_lzlo} we derive the upper bounds 
    \begin{align*}
        9 \sqrt{\frac{\Dz L_0}{\li(1-\momentum)}} 
        &\leq 9 \sqrt{\frac{\Dz L_0} \li} + 9U \\
        \frac{12 \Dz L_1}{\li (1-\momentum)}
        &\leq 12 \frac{\Dz L_1} \li + 7V \\
        \frac{5 \kappa }{\li (1-\momentum)} 
        &\leq 5 W \\
        5 \kappa (1-\momentum)^{\frac{p-1}p}
        &\leq 5 U + 7 V + 5 W
    \end{align*}
    on the remaining terms. Combining these bounds yields 
    \begin{equation*}
        \frac 1 \li \iSum \Exp{\dualnorm{\nFxt}}
        \leq 
        9 \sqrt{\frac{\Dz L_0} \li} + 12 \frac{\Dz L_1} \li + 14 U + 14 V + 10 W,
    \end{equation*}
    which finishes the proof.
\end{proof}

\subsubsection{Application to \texorpdfstring{\muon}{Muon}}
\label{sec:app.proofs.muon.muon}
Next we extend the previous result to \muon, which is an instance of \ourAlg\ with $\norm \cdot = \opnorm \cdot$, i.e.,
\begin{align}\tag{\muon}\begin{split}
    \mtp &\gets \momt \mtm + (1-\momt) \nfxt\\
    \xtp &\gets \xt + \sst \argmin_{\opnorm{d}} \langle \mt, d \rangle,
\end{split}\end{align}
where $\argmin_{\opnorm{d}} \langle \mt, d \rangle = UV$ is the polar decomposition.
In order to analyse \muon, we need to upper-bound the nuclear norm $\nucnorm \cdot \leq \rho \norm \cdot_p$ with a $p$-uniformly smooth norm $\norm \cdot_p$, to use the machinery from \Cref{sec:app.proofs.muon.general_norm}. In particular, we want to choose a $p$-uniformly smooth norm that has the smallest norm equivalence constant $\rho$ to $S_1$. We motivate our choice of $p$-uniformly smooth norm with the following lemma, which shows that the Schatten-$p$ norm is the best possible choice in the above sense.

\begin{lemma}
    Let $\norm \cdot$ be a $p$-uniformly smooth norm on $\R^{m \times n}$ with smoothness constant $S_p(\norm \cdot) \leq S$. Furthermore assume the normalisation $\norm {u v^\top} = 1$ for all $u \in \R^m, v \in \R^n$ with $\norm u_2 = \norm v_2 = 1$. Then we have
    \begin{equation*}
        \sup_{A \in \R^{m \times n}\setminus \set{0}} \frac{\nucnorm{A}}{\norm{A}} \geq \max\set{1, \frac{ \min\set{m,n}^{\frac {p-1} p}} S}.
    \end{equation*}
    In particular, the Schatten-$p$ norm achieves this lower-bound exactly.
\end{lemma}

\begin{proof}
    Denote $d \coloneqq \min\set{m,n}$ and define
    \begin{equation*}
        C \coloneqq \sup_{A \in \R^{m \times n}\setminus \set{0}} \frac{\nucnorm{A}}{\norm{A}}.
    \end{equation*}
    Furthermore let $E_i \coloneqq e_i f_i^\top \in \R^{m \times n}$, where $(e_i)_{i=1}^m, (f_i)_{i=1}^n$ are the standard bases of $\R^m$ and $\R^n$ respectively, for $i = 1, \dots, d$, and $\eps_1, \dots, \eps_d$ be independent Rademacher random variables. Define
    \begin{equation*}
        M_k \coloneqq \sum_{i=1}^k \eps_i E_i.
    \end{equation*}
    Then, by definition of the smoothness constant $S_p(\R^{m \times n})$ (see \Cref{def:constants}), we have
    \begin{equation*}
        \Exp[\eps_1, \dots, \eps_{k-1}]{\norm{M_k}^p} 
        = \frac{\norm{M_{k-1} + E_k}^p + \norm{M_{k-1} - E_k}^p}{2}
        \leq \norm{M_{k-1}}^p + S^p \norm{E_k}^p.
    \end{equation*}
    Hence, by induction and $\norm{E_k} = 1$, we get $\Exp{\norm{M_d}^p} \leq d S^p$. On the other hand note that 
    \begin{equation*}
        C \geq \frac{\nucnorm{M_d}}{\norm{M_d}} = \frac d {\norm{M_d}}
    \end{equation*}
    and hence
    \begin{equation*}
        \frac {d^p} {C^p} 
        \leq \Exp{\norm{M_d}^p}
        \leq d S^p.
    \end{equation*}
    Rearranging yields the claim.
\end{proof}

Finally we can state the $(L_0, L_1)$-smooth version of \Cref{thm:muon_convergence}.

\begin{corollary}\label{cor:convergence_muon}
    Suppose \Cref{assum:lower,assum:lzlosmooth,assum:noise} hold and let $\kappa \coloneqq \min \set{m, n}^{\frac{(p-1)}{p}}$. Then the iterates generated by \muon\ with 
    \begin{align*}
        \sst &\equiv \ssi = \min\set{\frac{1-\momentum}{5L_1}, \sqrt{\frac{\Delta_1(1-\momentum)}{L_0 T}}},\\
        \momt &\equiv \momentum = 1 - \min\set{1, \max \set{
            \pare{{\frac{\Dz L_0}{\kappa^2 \varsym^2 T}}}^{\frac{p}{3p-2}},
            \pare{\frac{5L_1\Delta_1}{2 \kappa \varsym T}}^{\frac{p}{2p-1}},
            \li^{-\frac{p}{2p-1}}
        }}
    \end{align*}
    satisfy
    \begin{align*}
        \frac 1 \li \iSum \Exp{\nucnorm{\nFxt}} 
        \leq &\ 
        9 \sqrt{\frac{\Dz L_0} \li} 
        + 12 \frac{\Dz L_1} \li 
        + 14 \pare{\frac{\min \set{m, n} \Dz L_0  \varsym^{\frac p {p-1}}} \li}^{\frac{p-1}{3p-2}} \\
        &\ + 14 \pare{\frac{\min \set{m, n} \Dz L_1 \varsym^{\frac p {p-1}}} \li}^{\frac{p-1}{2p-1}} 
        + 10 \frac{\min \set{m, n}^{\frac{(p-1)}{p}} \varsym}{\li^{\frac{p-1}{2p-1}}}.
    \end{align*}
    In particular, the sample complexity to reach an $\eps$-$\nucnorm \cdot$-stationary point is at most
    \begin{equation*}
        \Oc \pare{
            \frac{\Dz L_1}{\eps}
            + \frac{\Dz L_0}{\eps^2}
            + \iota \frac{\Dz L_1}{\eps}\pare{\frac{\varsym}{\eps}}^{\frac p {p-1}}
            + \iota \frac{\Dz L_0}{\eps^2}\pare{\frac{ \varsym}{\eps}}^{\frac p {p-1}}
            + \iota^{\frac{(2p-1)}{p}}\pare{\frac{\varsym}{\eps}}^{\frac{2p-1}{p-1}}
        },
    \end{equation*}
    where $\iota \coloneqq {\min \set{m, n}}$.
\end{corollary}

\begin{proof}
    By \Cref{lem:schatten_norm_equivalence_constants} we have $\nucnorm{A} \leq \kappa \norm{A}_{S_p}$ for all $A \in \R^{m \times n}$, and the Schatten-$p$ norm is $p$-uniformly smooth with smoothness constant $S_p(\R^{m \times n}) \leq 1$ (\Cref{example:schatten_p}).
    Hence we may apply \Cref{thm:convergence_nssdm_norm_equivalence_formulation} with $(\Bc, \norm \cdot) = (\R^{m \times n}, \opnorm \cdot), \norm \cdot_p = \norm \cdot_{S_p}, \rho \gets \kappa$ and $S = 1$. 
\end{proof}

\begin{remark}\label{rem:application_full_network_including_vectors}
    Extending \Cref{rem:full_network_muon}, we note that \Cref{thm:convergence_nssdm_norm_equivalence_formulation} also extends to full networks including vector valued weights by considering \citep{SCION2025pethick} the Banach space
    \begin{align*}
        \Bc &= \R^{m_1, n_1} \times \hdots \times \R^{m_D, n_D} \times \R^{d_1} \times \hdots \times \R^{d_K},\\
        \norm \cdot &= \max\set{\max_{l \in [D]} \opnorm{W_l}, \max_{l \in [K]} \norm{v_l}_{\ell_\infty}}.
    \end{align*}
    In this case, vector-valued weights are trained with \textsc{Signum}, which is widely considered to be an idealised version of \textsc{Adam} \citep{kunstner2023noise}. The corresponding dimension dependence is given by
    \begin{equation*}
        \sum_{l \in [D]} \min\set{m_l, n_l} + \sum_{l \in [K]} d_l.
    \end{equation*}
\end{remark}

\section{Extension to Stochastic Conditional Gradient}
\label{sec:app.proofs.scg}
Let $\mathcal D \subset \Bc$ be convex and bounded, with diameter
\begin{equation*}
D_{\mathcal D} \coloneqq \sup_{x,y \in \mathcal D} \norm{x-y}.
\end{equation*}
Consider
\begin{equation*}
\min_{x \in \mathcal D} F(x),
\end{equation*}
and assume that the linear minimisation oracle over $\mathcal D$ is exact, i.e., each subproblem $\inf_{u \in \mathcal D}\angles{\phi,u}$ admits a minimiser. We consider the stochastic conditional gradient (SCG) iteration
\begin{equation}
\tag{SCG}\label{eq:SCG}
\begin{aligned}
m_t &\gets \beta_t m_{t-1} + (1-\beta_t)\nf\pare{x_t,\xi_t}, \\
v_t &\in \argmin_{u \in \mathcal D} \angles{m_t, u}, \\
x_{t+1} &\gets (1-\eta_t)x_t + \eta_t v_t,
\end{aligned}
\end{equation}
where $\iterationm 0 \coloneqq \nf \pare{\iterationx \fin, \iterationxi \fin}$. We measure stationarity by the Frank-Wolfe gap
\begin{equation*}
g_{\mathcal D}(x) \coloneqq \sup_{u \in \mathcal D} \angles{\nabla F(x), x-u}.
\end{equation*}

\begin{lemma}[Constrained descent bound]
\label{lem:scg_descent}
Suppose \lsmooth\ holds. Let $\mu_t \coloneqq m_t - \nabla F(x_t)$. Then, for every $t \geq 1$,
\begin{equation*}
\eta_t \Exp{g_{\mathcal D}(x_t)}
\leq
\Exp{F(x_t)-F(x_{t+1})}
+ \eta_t D_{\mathcal D}\Exp{\dualnorm{\mu_t}}
+ \tfrac L 2 \eta_t^2 D_{\mathcal D}^2.
\end{equation*}
\end{lemma}
\begin{proof}
For any $u \in \mathcal D$, smoothness and $x_{t+1} = x_t + \eta_t(v_t-x_t)$ give
\begin{align*}
F(x_{t+1})
&\leq F(x_t) + \eta_t \angles{\nabla F(x_t), v_t-x_t} + \tfrac L 2 \eta_t^2 \norm{v_t-x_t}^2 \\
&\leq F(x_t) + \eta_t \angles{\nabla F(x_t), u-x_t} + \eta_t \angles{\mu_t, u-v_t} + \tfrac L 2 \eta_t^2 \norm{v_t-x_t}^2,
\end{align*}
where we used $\angles{m_t,v_t} \leq \angles{m_t,u}$. Since $u,v_t,x_t \in \mathcal D$,
\begin{equation*}
\eta_t \angles{\nabla F(x_t), x_t-u}
\leq
F(x_t)-F(x_{t+1}) + \eta_t D_{\mathcal D}\dualnorm{\mu_t} + \tfrac L 2 \eta_t^2 D_{\mathcal D}^2.
\end{equation*}
Taking expectations and the supremum over $u$ proves the claim.
\end{proof}

The proof relies on an error control Lemma analogue to \Cref{lem:mom:deviation_bound}.
Fortunately, the argument in \Cref{lem:mom:deviation_bound} only depends on the momentum recursion in the update rule, which remains unchanged for \ref{eq:SCG}.
The only new ingredient in the following proof is that the iterate difference is controlled by the domain diameter rather than the unconstrained update magnitude.

\begin{lemma}[Momentum error for \ref{eq:SCG}]
\label{lem:scg_error}
Suppose \Cref{assum:lsmooth,assum:noise} hold, and that $(\Bc,\norm\cdot)$ is $q$-uniformly convex with $q$-convexity constant $K$. Let $p'$ be the conjugate exponent of $q$ and set $r \coloneqq \min\set{p,p'}$. For constant parameters $\eta_t \equiv \eta$, and $\beta_t \equiv \beta \in [0,1)$,
\begin{equation*}
\textstyle\sum_{t=1}^T \Exp{\dualnorm{m_t - \nabla F(x_t)}}
\leq
\varsym K \tfrac{1}{1-\beta}
+ \varsym K T (1-\beta)^{\tfrac{r-1}{r}}
+ \tfrac{\eta L D_{\mathcal D} T}{1-\beta}.
\end{equation*}
\end{lemma}
\begin{proof}
Writing $\delta_t \coloneqq \nf\pare{x_t,\xi_t} - \nabla F(x_t)$, we have
\begin{equation*}
\mu_t
:=
m_t - \nabla F(x_t)
=
\beta_t \mu_{t-1} + (1-\beta_t)\delta_t + \beta_t\pare{\nabla F(x_{t-1})-\nabla F(x_t)}.
\end{equation*}
For constant $\beta$, unrolling gives
\begin{equation*}
\mu_t
=
\sum_{\tau=1}^t b_{\tau,t}\delta_\tau
+ \sum_{\tau=2}^t c_{\tau,t}\pare{\nabla F(x_{\tau-1})-\nabla F(x_\tau)},
\end{equation*}
where $b_{1,t}=\beta^{t-1}$, $b_{\tau,t}=(1-\beta)\beta^{t-\tau}$ for $\tau \geq 2$, and $c_{\tau,t}=\beta^{t-\tau+1}$.
By \lsmooth,
\begin{equation*}
\dualnorm{\nabla F(x_{\tau-1})-\nabla F(x_\tau)}
\leq
L \norm{x_{\tau-1}-x_\tau}.
\end{equation*}
The \ref{eq:SCG} update implies $x_t-x_{t-1}=\eta(v_{t-1}-x_{t-1})$ with $x_{t-1},v_{t-1}\in\mathcal D$, hence
\begin{equation*}
\norm{x_t-x_{t-1}} \leq \eta D_{\mathcal D}.
\end{equation*}
Therefore
\begin{equation*}
\Exp{\dualnorm{\mu_t}}
\leq
\Exp{\dualnorm{\sum_{\tau=1}^t b_{\tau,t}\delta_\tau}}
+ \eta L D_{\mathcal D} \sum_{\tau=2}^t c_{\tau,t}.
\end{equation*}
Since $(\delta_\tau)_{\tau=1}^t$ is a martingale difference sequence in $\Bs$, the same application of \Cref{thm:banach_martingale,lem:make_smoothness_weaker} as in \Cref{lem:mom:deviation_bound} yields
\begin{equation*}
\Exp{\dualnorm{\sum_{\tau=1}^t b_{\tau,t}\delta_\tau}}
\leq
\varsym K \pare{\sum_{\tau=1}^t b_{\tau,t}^r}^{1/r}.
\end{equation*}
Summing over $t$ and using the same geometric-series estimates as in \Cref{lem:mom:deviation_bound} proves the claim.
\end{proof}

We are now ready to state the convergence theorem for \ref{eq:SCG}.

\begin{theorem}[Frank-Wolfe gap bound for \ref{eq:SCG}]
\label{thm:scg_gap}
Suppose \Cref{assum:lsmooth,assum:noise} hold, and that $(\Bc,\norm\cdot)$ is $q$-uniformly convex with $q$-convexity constant $K$. 
Let $p'$ be the conjugate exponent of $q$, $r \coloneqq \min\set{p,p'}$, and
\begin{equation*}
\Delta_{\mathcal D} \coloneqq F(x_1) - \inf_{x \in \mathcal D} F(x).
\end{equation*}
Consider \ref{eq:SCG} with
\begin{equation*}
\begin{split}
&\eta_t \equiv \eta \coloneqq \min\set{1,\sqrt{\tfrac{\Delta_{\mathcal D}(1-\beta)}{L D_{\mathcal D}^2 T}}}, \\
&\beta_t \equiv \beta \coloneqq
1-\min\set{1,\max\set{
\pare{\tfrac{\Delta_{\mathcal D}L}{\varsym^2 K^2 T}}^{\tfrac{r}{3r-2}},
T^{-\tfrac{r}{2r-1}}
}}.
\end{split}
\end{equation*}
Then
\begin{equation*}
\tfrac 1 T \textstyle\sum_{t=1}^T \Exp{g_{\mathcal D}(x_t)}
\leq
\tfrac{\Delta_{\mathcal D}}{T}
+ 3 D_{\mathcal D}\sqrt{\tfrac{\Delta_{\mathcal D}L}{T}}
+ 3 D_{\mathcal D}\pare{\tfrac{\Delta_{\mathcal D}L(\varsym K)^{\tfrac{r}{r-1}}}{T}}^{\tfrac{r-1}{3r-2}}
+ 2 D_{\mathcal D}\tfrac{\varsym K}{T^{\tfrac{r-1}{2r-1}}}.
\end{equation*}
\end{theorem}
\begin{proof}
Summing \Cref{lem:scg_descent} and telescoping gives
\begin{equation*}
\eta \textstyle\sum_{t=1}^T \Exp{g_{\mathcal D}(x_t)}
\leq
\Delta_{\mathcal D}
+ \eta D_{\mathcal D}\textstyle\sum_{t=1}^T \Exp{\dualnorm{\mu_t}}
+ \tfrac L 2 \eta^2 D_{\mathcal D}^2 T.
\end{equation*}
Plugging in \Cref{lem:scg_error} yields
\begin{equation*}
\tfrac 1 T \textstyle\sum_{t=1}^T \Exp{g_{\mathcal D}(x_t)}
\leq
\tfrac{\Delta_{\mathcal D}}{\eta T}
+ D_{\mathcal D}\tfrac{\varsym K }{T(1-\beta)}
+ D_{\mathcal D}\varsym K (1-\beta)^{\tfrac{r-1}{r}}
+ \tfrac{\eta L D_{\mathcal D}^2}{1-\beta}
+ \tfrac{L \eta D_{\mathcal D}^2}{2}.
\end{equation*}
Let $\bar\eta \coloneqq \sqrt{\tfrac{\Delta_{\mathcal D}(1-\beta)}{L D_{\mathcal D}^2 T}}$, so that $\eta=\min\set{1,\bar\eta}$. 
We have that $\frac{1}{\eta} \leq 1 + \frac{1}{\bar\eta}$ and $\eta \leq \bar\eta$.
Consequently,
\begin{equation*}
\tfrac{\Delta_{\mathcal D}}{\eta T}
\leq
\tfrac{\Delta_{\mathcal D}}{T}
+ D_{\mathcal D}\sqrt{\tfrac{\Delta_{\mathcal D}L}{T(1-\beta)}},
\quad
\tfrac{\eta L D_{\mathcal D}^2}{1-\beta}
\leq
D_{\mathcal D}\sqrt{\tfrac{\Delta_{\mathcal D}L}{T(1-\beta)}},
\quad
\tfrac{L \eta D_{\mathcal D}^2}{2}
\leq
\tfrac{D_{\mathcal D}}{2}\sqrt{\tfrac{\Delta_{\mathcal D}L(1-\beta)}{T}}
\leq
\tfrac{D_{\mathcal D}}{2}\sqrt{\tfrac{\Delta_{\mathcal D}L}{T}}.
\end{equation*}
Next, let
\begin{equation*}
U \coloneqq \pare{\tfrac{\Delta_{\mathcal D}L(\varsym K)^{\tfrac{r}{r-1}}}{T}}^{\tfrac{r-1}{3r-2}},
\qquad 
W \coloneqq \tfrac{\varsym K}{T^{\tfrac{r-1}{2r-1}}},
\qquad
u \coloneqq \pare{\tfrac{\Delta_{\mathcal D}L}{\varsym^2 K^2 T}}^{\tfrac{r}{3r-2}},
\qquad
w \coloneqq T^{-\tfrac{r}{2r-1}},
\end{equation*}
so that $1-\beta = \min\set{1,\max\set{u,w}}$. Hence
\begin{equation*}
\tfrac{1}{\sqrt{1-\beta}}
\leq
\tfrac{1}{\min\set{1,\sqrt u}}
\leq
1 + \tfrac{1}{\sqrt u},
\end{equation*}
which implies
\begin{equation*}
2D_{\mathcal D}\sqrt{\tfrac{\Delta_{\mathcal D}L}{T(1-\beta)}}
+ \tfrac{D_{\mathcal D}}{2}\sqrt{\tfrac{\Delta_{\mathcal D}L}{T}}
\leq
3D_{\mathcal D}\sqrt{\tfrac{\Delta_{\mathcal D}L}{T}} + 2D_{\mathcal D}U.
\end{equation*}
Moreover, since $1-\beta \geq w$ and $\beta \leq 1$,
\begin{equation*}
D_{\mathcal D}\tfrac{\varsym K }{T(1-\beta)}
\leq
D_{\mathcal D}W,
\end{equation*}
and, using $(1-\beta)^\alpha \leq u^\alpha + w^\alpha$ for $\alpha = \tfrac{r-1}{r}$,
\begin{equation*}
D_{\mathcal D}\varsym K (1-\beta)^{\tfrac{r-1}{r}}
\leq
D_{\mathcal D}U + D_{\mathcal D}W.
\end{equation*}
Combining the bounds gives
\begin{equation*}
\tfrac 1 T \textstyle\sum_{t=1}^T \Exp{g_{\mathcal D}(x_t)}
\leq
\tfrac{\Delta_{\mathcal D}}{T}
+ 3D_{\mathcal D}\sqrt{\tfrac{\Delta_{\mathcal D}L}{T}}
+ 3D_{\mathcal D}U
+ 2D_{\mathcal D}W,
\end{equation*}
which is exactly the stated claim.
\end{proof}

\section{Lower Bound Proofs}
\label{sec:app.lower_bounds}
In this section we carry the missing proof of \Cref{thm:lower_bound}. Therefore let us recite the definitions of \Cref{sec:main.lb} for the reader's convenience. Let $m, n \in \N_{\geq 2}$ and set
\begin{equation*}
    \kappa \coloneqq \min\set{m, n} \geq 2
\end{equation*}
Further let $(e_i)_{i=1}^m$ and $(f_j)_{j=1}^n$ are the standard bases of $\R^m$ and $\R^n$ respectively, and 
\begin{equation*}
    E_i \coloneqq e_i f_i^\top \in \R^{m \times n}
\end{equation*}
For a sign vector $s \in \set{\pm 1}^\kappa$, define
\begin{equation*}
    S_s \coloneqq \sum_{i = 1}^\kappa s_i E_i.
\end{equation*}
Fix $\delta > 0$ to be specified later and $q \coloneqq \pare{\frac{2 \delta}{\varsym}}^{\frac{p}{p-1}}$. Then the hard instance is given by
\begin{equation*}
    F_s(x) \coloneqq \frac{L}{2\kappa} \fnorm{x}^2 - \frac{\delta}{\kappa} \langle S_s, x\rangle_F, \qquad
    \nabla f_s (x, \xi) 
    \frac L \kappa x - \frac{\delta}{q} B s_I E_I,
\end{equation*}
where $B \sim \operatorname{Bern}(q)$ and $I \sim \operatorname{Unif}([\kappa])$ independently.

Now we first verify that $F_s$ and $\nabla f_s$ satisfy the required properties.
\begin{lemma}\label{lem:hard_instance_properties}
    Assume $\delta \leq \nicefrac \sigma 2$. Then the function $F_s$ is $L$-smooth with $F(0) - \inf_{x} F_s(x) \leq \frac{\delta^2}{2L}$ and the gradient oracle $\nabla f_s$ is unbiased and satisfies $\Exp{\nucnorm{\nabla f_s (x, \xi) - \nabla F_s(x)}^p} \leq \sigma^p$.
\end{lemma}
\begin{proof}
    The smoothness is immediate by noting that $\nucnorm{\nabla F_s(x) - \nabla F_s(y)} = \frac L \kappa \nucnorm{x - y} \leq L \opnorm{x-y}$ for all $x, y \in \R^{m \times n}$. To prove the bounded initialisation gap, we complete the square to get
    \begin{equation*}
        F_s(x) = \frac{L}{2\kappa} \fnorm{x - \frac{\delta}{L} S_s}^2 - \frac{\delta^2}{2L\kappa} \fnorm{S_s},
    \end{equation*}
    and hence $F_s(x) \geq - \frac{\delta^2}{2L}$. Unbiasedness again directly follows from the independence of $B$ and $I$ and the definition of $q$. Finally, we have
    \begin{equation*}
        \Exp{\nucnorm{\nabla f(x, \xi) - \nabla F(x)}^p}
        = \Exp{\nucnorm{\frac \delta \kappa S_s - \frac{\delta}{q} B s_I E_I}^p}
        \leq 2^{p-1} \pare{\delta^p + \Exp{\nucnorm{\frac \delta q B s_I E_I}^p}}
    \end{equation*}
    and
    \begin{equation*}
        \Exp{\nucnorm{\frac \delta q B s_I E_I}^p}
        = q \pare{\frac{\delta} q}^p\Exp{\nucnorm{s_I E_I}^p}
        = q \pare{\frac{\delta} q}^p 
        = \frac{\delta^p}{q^{p-1}}
        = \frac{\sigma^p}{2^p}.
    \end{equation*}
    Hence the claim follows from our assumption $\delta \leq \frac{\sigma}{2}$.
\end{proof}

Additionally we need the following technical lemmata.

\begin{lemma}\label{lem:von_Neuman_rectangular_matricies}
    For every $A \in \R^{m \times n}$, $\nucnorm{A} \geq \sum_{i = 1}^\kappa \abs{A_{ii}}$.
\end{lemma}
\begin{proof}
    Let $\omega_i \coloneqq \sgn(A_{ii})$ for $i = 1, \ldots, \kappa$ and $U \coloneqq \sum_{i=1}^\kappa \omega_i E_i$. Then $\opnorm{U} \leq 1$ and, by duality of the nuclear and operator norm,
    $
        \sum_{i=1}^\kappa \abs{A_{ii}} 
        = \langle U, A\rangle_F 
        \leq \opnorm{U}\nucnorm{A}
        = \nucnorm A.
    $
\end{proof}

For a sign vector $s \in \set{\pm 1}^\kappa$ we denote the version which flips the sign at index $i$ as $s^{(i)} = (s_1, \ldots, s_{i-1}, -s_i, s_{i+1}, \ldots, s_\kappa)$. The following lemma says that, as long as $(B_t, I_t) = (1, i)$ did not yet happen, the iterates of an algorithm run on $F_s$ and $F_{s^{(i)}}$ are identical.

\begin{lemma}\label{lem:undiscovered_indices_indistinguishable}
    Let $\mathcal A$ be an optimization algorithm, $i \in [\kappa]$ and $s \in \set{\pm 1}^\kappa$. Let $B_1, \dots, B_T$ and $I_1, \dots, I_T$ be independent and run $\Ac$ on $F_s$ and $F_{s^{(i)}}$ to obtain iterates $x_1, g_1, \dots, x_T, g_T$ and $x_1', g_1', \dots, x_T', g_T'$ respectively. Furthermore let 
    \begin{equation*}
        A_i \coloneqq \bigcap_{t=1}^T \set{(I_t, B_t) \neq (i, 1)}.
    \end{equation*}
    Then we have $\mathds{1}_{A_i} (x_t, g_t) = \mathds{1}_{A_i} (x_t', g_t')$ for all $t \in [T]$.
\end{lemma}
\begin{proof}
    This follows by induction. For $t = 1$ the claim is trivial as $x_1 = x_1'$. Now assume the claim holds for some $t \in [T-1]$. Then, on the event $A_i$, we have $\nabla f_s(x_t, \xi_t) = \nabla f_{s^{(i)}}(x_t', \xi_t)$ and hence $x_{t+1} = x_{t+1}'$.
\end{proof}

Finally we need a technical lemma which roughly states that, for $s \sim \unif\pare{\set{\pm 1}^\kappa}$, as long as the algorithm has not seen an index $i$, $s_i$ is still uniformly distributed conditioned on the history.

\begin{lemma}\label{lem:still_uniform_if_undiscovered}
    Let $\Ac$ be an optimization algorithm with internal randomness $r$. Let $s \sim \unif\pare{\set{\pm 1}^\kappa}$ and denote the sequence generated by $\Ac$ on $F_s$ as $(x_1, g_1), \dots, (x_T, g_T)$. Furthermore define the observed randomness $\mathcal H_T \coloneqq \sigma\pare{\set{r, g_1, \dots, g_T}}$ and 
    \begin{equation*}
        A_i \coloneqq \bigcap_{t=1}^T \set{(I_t, B_t) \neq (i, 1)}.
    \end{equation*}
    Finally assume that $r, s, B_1, I_1, \dots, B_T, I_T$ are mutually independent. Then
    \begin{equation*}
        \mathds{1}_{A_i} \Exp[\mathcal H_T]{s_i} = 0 \qquad \text{a.s.}
    \end{equation*}
    In particular, for any $\mathcal H_\li$-measurable $\nu$ and Borel-measurable function $\psi \colon \R \times \set{\pm 1} \to \R$, we have
    \begin{equation*}
        \mathds{1}_{A_i} \Exp[\mathcal H_T]{\psi(\nu, s_i)}
        = \mathds{1}_{A_i} \frac{\psi(\nu, 1) + \psi(\nu, -1)}{2}.
    \end{equation*}
\end{lemma}
\begin{proof}
    Fix $i \in [\kappa]$, and define the sigma-algebra containing all primitive randomness except $s_i$ as
    \begin{equation*}
        \mathcal K_i \coloneqq \sigma \pare{r, I_1, B_1, \dots, I_T, B_T, s_1, \dots, s_{i-1}, s_{i+1}, \dots, s_\kappa}.
    \end{equation*}
    Furthermore let $M \coloneqq \pare{r, g_1, \dots, g_T}$ be the trajectory on $F_s$, and $M^{(\pm 1)}$ the counterfactual trajectories on $F_{s^{(\pm 1)}}$, where $s^{(\pm 1)} = (s_1, \dots, s_{i-1}, \pm 1, s_{i+1}, \dots, s_\kappa)$.

    Now let $Z$ be any bounded $\mathcal H_\li$-measurable random variable. Since $\mathcal H_\li = \sigma(M)$, the Doob-Dynkin Lemma supplies a bounded Borel-measurable function $h \colon \R \to \R$ such that $Z = h(M)$. Hence
    \begin{equation*}
        \Exp{\mathds{1}_{A_i} Z s_i}
        = \Exp{\Exp[\mathcal K_i]{\mathds{1}_{A_i} Z s_i}}
        = \Exp{\mathds{1}_{A_i} \frac{h(M^{(1)}) - h(M^{(-1)})}{2}},
    \end{equation*}
    where we used $A_i \in \mathcal K_i, M = M^{(s_i)}$ and that $s_i$ is independent of $\mathcal K_i$. By \Cref{lem:undiscovered_indices_indistinguishable} we have $\mathds{1}_{A_i} M^{(1)} = \mathds{1}_{A_i} M^{(-1)}$ and hence $\Exp{\mathds{1}_{A_i} Z s_i} = 0$. The definition of conditional expectation hence implies $\Exp[\mathcal H_\li]{\mathds{1}_{A_i} s_i} = 0$, which finishes the first part.

    For the second statement note that $\psi(\nu, x) = \alpha(\nu) + \beta(\nu) x$, where
    \begin{equation*}
        \alpha(\nu) = \frac{\psi(\nu, 1) + \psi(\nu, -1)}{2}, \qquad 
        \beta(\nu) = \frac{\psi(\nu, 1) - \psi(\nu, -1)}{2}
    \end{equation*}
    and thus
    \begin{equation*}
        \mathds{1}_{A_i} \Exp[\mathcal H_T]{\psi(\nu, s_i)}
        = \mathds{1}_{A_i} \alpha(\nu) + \beta(\nu) \Exp[\mathcal H_T]{\mathds{1}_{A_i} s_i}
        = \mathds{1}_{A_i} \alpha(\nu),
    \end{equation*}
    where we used the previous result in the last step. This finishes the proof.
\end{proof}

This finally allows us to provide the lower-bound on $\Exp{\nucnorm{\nabla F_s\pare{x_T}}}$. Importantly, we will choose $s$ to be uniformly distributed over $\set{\pm 1}^\kappa$, which allows us to leverage the fact that some $F_s$ are hard to differentiate as highlighted by \Cref{lem:undiscovered_indices_indistinguishable}.

\begin{proposition}\label{prop:generic_lower_bound}
    Let $s \sim \unif\pare{\set{\pm 1}^\kappa}$. Then, for every possibly randomised algorithm $\Ac$ and every $T \in \Ngeq$, we have
    \begin{equation*}
        \Exp{\nucnorm{\nabla F_s\pare{x_T}}} 
        \geq \delta \pare{1 - \frac{q}{\kappa}}^T.
    \end{equation*}
\end{proposition}
\begin{proof}
    Denote the sequence generated by $\Ac$ on $F_s$ as $(x_1, g_1), \dots, (x_T, g_T)$ and define $\mathcal H_T$ as in the previous lemma. For $i \in [\kappa]$ denote $\nu_i \coloneqq \langle E_i, x_T \rangle$ and note that $\nu_i$ is $\mathcal H_T$-measurable. Moreover, we have
    \begin{equation*}
        \langle E_i, \nabla F_s( x_T) \rangle
        = \frac{L \nu_i - \delta s_i}{\kappa}
    \end{equation*}
    and hence, by \Cref{lem:von_Neuman_rectangular_matricies},
    \begin{equation*}
        \nucnorm{\nabla F_s \pare{x_T}}
        \geq \sum_{i=1}^\kappa \abs{\langle E_i, \nabla F_s(x_T) \rangle}
        = \sum_{i=1}^\kappa \frac{\abs{L \nu_i - \delta s_i}}{\kappa}.
    \end{equation*}
    Taking expectation we obtain
    \begin{equation*}
        \Exp{\nucnorm{\nabla F_s \pare{x_T}}}
        \geq \sum_{i=1}^\kappa \Exp{\frac{\abs{L \nu_i - \delta s_i}}{\kappa}}
        \geq \sum_{i=1}^\kappa \Exp{\mathds{1}_{A_i} \frac{\abs{L \nu_i - \delta s_i}}{\kappa}}
    \end{equation*}
    and, by \Cref{lem:still_uniform_if_undiscovered},
    \begin{equation*}
        \Exp{\mathds{1}_{A_i} \frac{\abs{L \nu_i - \delta s_i}}{\kappa}}
        = \Exp{\mathds{1}_{A_i} \Exp[\mathcal H_T]{\frac{\abs{L \nu_i - \delta s_i}}{\kappa}}}
        \geq \Exp{\mathds{1}_{A_i} \frac{\abs{L \nu_i - \delta} + \abs{L\nu_i + \delta}}{2\kappa}}
    \end{equation*}
    Finally, note that $\frac{\abs{L \nu_i - \delta} + \abs{L\nu_i + \delta}}{2\kappa} \geq \frac{\delta}{\kappa}$, and $\Exp{\mathds{1}_{A_i}} = \Pr(A_i) = (1 - \frac q \kappa)^T$ to close the proof.
\end{proof}

Finally we provide the proof of \Cref{thm:lower_bound}.

\begin{proof}[Proof of \Cref{thm:lower_bound}]
    Let $\delta \coloneqq 2 \eps$ which satisfies $\delta \leq \frac \sigma 2$ by our assumption on $\eps$. By \Cref{lem:hard_instance_properties}, the function $F_s$ satisfies
    \begin{equation*}
        F(0) - \inf_{x \in \R^{m \times n}} F_s(x) \leq \frac{\delta^2}{2L} = \frac{2\eps^2} L \leq \Dz
    \end{equation*}
    by our assumption on $\eps$. Hence $F_s \in \Fc_{\Dz, L}$ and $\nabla f_s \in \Oc_{p, \sigma}$. Furthermore, for $s \sim \unif\pare{\set{\pm 1}^\kappa}$, \Cref{prop:generic_lower_bound} implies
    \begin{equation*}
        \Exp{\nucnorm{\nabla F_s\pare{\hat x_T}}} 
        \geq \delta \pare{1 - \frac{q}{\kappa}}^T
    \end{equation*}
    and hence, for the output to satisfy $\Exp{\nucnorm{\nabla F_s\pare{\hat x_T}}} \leq \eps$, we need $2 \eps \pare{1 - \nicefrac q \kappa}^T \leq \eps$. This is equivalent to
    \begin{equation*}
        \pare{1 - \frac q \kappa}^T \leq \frac 1 2
        \Leftrightarrow T \log\pare{1 - \frac q \kappa} \leq \log\pare{\frac 1 2}
        \Leftrightarrow T \geq \frac{\log(2)}{-\log\pare{1 - \frac q \kappa}}
    \end{equation*}
    Since $\kappa \geq 2$ we have $\nicefrac q \kappa \leq \nicefrac 1 2$ and thus we can use $-\log(1-x) \leq 2x \log(2)$, for $x \in [0, \nicefrac 1 2]$, to get
    \begin{equation*}
        -\log\pare{1 - \frac q \kappa} 
        \leq 2 \log(2) \frac{q}{\kappa}.
    \end{equation*}
    Combining yields
    \begin{equation*}
        T \geq
        \frac{\log\pare 2 \kappa }{2 \log \pare 2 q}
        = \frac{\kappa} 2 \cdot \pare{\frac{\sigma}{4\eps}}^{\frac p {p-1}},
    \end{equation*}
    finishing the proof.
\end{proof}

\subsection{High-Dimensional Lower-Bound}
\label{sec:app.lower_bounds.high_dim}
In this section we provide a first-order lower-bound for $\nucnorm \cdot$-stationarity of the form
\begin{equation*}
    \Omega\pare{\min\set{m, n} \frac{\Dz L}{\eps^2} \pare{\frac{\varsym}{\eps}}^{\frac{p}{p-1}}}.
\end{equation*}
As vector valued functions are a special case, it is well-known such lower-bounds require a high-dimensional construction \citep{LowerBoundsNonConvex2023arjevani}. Let us first recite the construction for the Euclidean case on $\R^d$ from \citet{LowerBoundsNonConvex2023arjevani,WhyAreAdaptive2020zhang}. For simplicity we will only consider the case of zero-respecting algorithms, the extension to randomised algorithms follows from the same techniques as in \citep{LowerBoundsNonConvex2023arjevani}.

Therefore we first introduce the Euclidean $\R^d$-version of the function class as
\begin{equation*}
    \Fc_{\Dz, L}^{(2)}(d) \coloneqq \set{F \colon \R^d \to \R \suchthat \norm{\nabla F(x) - \nabla F(y)}_2 \leq L \norm{x - y}_2, F(0) - \inf_{x} F(x) \leq \Dz}
\end{equation*}
and the Euclidean $\R^d$-version of the oracle class as
\begin{equation*}
    \Oc_{p, \varsym}^{(2)}(d) \coloneqq \set{\nabla f \colon \R^d \times \Xi \to \R^d \suchthat 
    \begin{gathered}
        \Exp{\nabla f(x, \xi)} = \nabla F(x), \\
        \Exp{\norm{\nabla f(x, \xi) - \nabla F(x)}_2^p} \leq \varsym^p, \forall x \in \R^d
    \end{gathered}
    }.
\end{equation*}

\begin{lemma}[\cite{LowerBoundsNonConvex2023arjevani,WhyAreAdaptive2020zhang}]\label{lem:eucl_lower_bound}
    Let $c \coloneqq 1 / 300$ and $p \in (1,2]$. For all $\Dz, L, \varsym > 0$ and $0 < \eps \leq c \min\set{\varsym \sqrt{\Dz L}}$ there exists a dimension $d \leq \Dz L \eps^{-2}$, a function $F \in \Fc_{\Dz, L}^{(2)}(d)$ and a gradient oracle $\nabla f \in \Oc_{p, \varsym}^{(2)}(d)$ such that every possibly randomised zero-respecting algorithm $\Ac$ satisfies
    \begin{equation*}
        \Exp{\norm{\nabla F(x_t)}_2} > \eps \qquad \text{for all} \qquad t \leq c^2\frac{ \Dz L}{\eps^2}\pare{\frac{c \varsym}{\eps}}^{\frac{p}{p-1}}.
    \end{equation*}
    In particular, $\nabla f(x, \xi)$ has the structure
    \begin{equation*}
        \nabla f(x, \xi) = \nabla F(x) + \pare{\frac{Z}{\theta} - 1} h(x),
    \end{equation*}
    where $\norm{h(x)}_2 \leq 46 \eps$, and $\xi \sim \operatorname{Bernoulli}(\theta)$ with $\theta = \pare{\frac \eps {c \varsym}}^{\frac{p}{p-1}}$.
\end{lemma}

For $(\R^{m \times n}, \opnorm \cdot)$, the corresponding classes are defined as
\begin{align*}
    \Fc_{\Dz, L}(m, n)
    &\coloneqq \set{F \colon \R^{m\times n} \to \R
    \suchthat
    \begin{gathered}
        F(0) - \inf_{X} F(X) \leq \Dz,\\
        \nucnorm{\nabla F(X) - \nabla F(Y)} \leq L \opnorm{X - Y}
    \end{gathered}
    },\\
    \Oc_{p, \varsym}(m, n)
    &\coloneqq \set{\nabla f \colon \R^{m\times n} \times \Xi \to \R^{m\times n}
    \suchthat
    \begin{gathered}
        \Exp{\nabla f(X, \xi)} = \nabla F(X),\\
        \Exp{\nucnorm{\nabla f(X, \xi) - \nabla F(X)}^p} \leq \varsym^p
    \end{gathered}
    }.
\end{align*}
Next we define the quantities which will be used to embed the hard instance from $\R^d$ into $\R^{m \times n}$. Therefore let $d \in \Ngeq$, $\kappa \in \Ngeq$, and define $m \coloneqq \kappa d$ and $n \coloneqq \kappa$. Now, for $i \in [\kappa]$ and $v \in \R^d$, let
\begin{equation*}
    R_i \coloneqq \set{(i-1) d + 1, \dots, id}, \qquad
    E_i(v) \coloneqq \sum_{j \in R_i} v_{j - (i-1)d} \cdot e_j f_i^\top \in \R^{m \times n}.
\end{equation*}
That is, $E_i(v)$ places $v$ in the rows $R_i$ of its $i$-th column and is zero everywhere else. Furthermore let $P_i$ be the inverse operation, i.e.,
\begin{equation*}
    P_i(A) \coloneqq \sum_{j \in R_i} \langle e_j f_i^\top, A\rangle e_{j - (i-1)d} \in \R^d
\end{equation*}
extracts the rows $R_i$ of the $i$-th column of $A$. Then the following properties hold.

\begin{lemma}\label{lem:embedding_properties}
    Let $v_1, \dots, v_\kappa \in \R^d$, and $A \in \R^{m \times n}$. Then
    \begin{align*}
        \opnorm{\sum_{i = 1}^\kappa E_i(v_i)} = \max_{i \in [\kappa]} \norm{v_i}_2, \quad
        \nucnorm{\sum_{i = 1}^\kappa E_i(v_i)} = \sum_{i \in [\kappa]} \norm{v_i}_2, \quad \text{and} \quad
        \norm{P_i(A)}_2 \leq \opnorm{A}. 
    \end{align*}
\end{lemma}
\begin{proof}
    The columns of $\sum_{i = 1}^\kappa E_i(v_i)$ have disjoint row support, so
    \begin{equation*}
        \pare{\sum_{i = 1}^\kappa E_i(v_i)}^\top \pare{\sum_{i = 1}^\kappa E_i(v_i)} = \operatorname{Diag}\pare{\norm{v_1}_2^2, \dots, \norm{v_\kappa}_2^2}
    \end{equation*}
    and the singular values are hence $\norm{v_1}_2, \dots, \norm{v_\kappa}_2$, which proves the first two identities. For the last claim note that $\norm{P_i(A)}_2 \leq \norm{Af_i}_2 \leq \opnorm{A}$.
\end{proof}

This allows us to embed the hard instance from $\R^d$ into $\R^{m \times n}$ as follows. Let $F^{(2)} \in \Fc_{\Dz^{(2)}, L^{(2)}}^{(2)}(d)$ and $\nabla f^{(2)} \in \Oc_{p, \varsym^{(2)}}^{(2)}(d)$ be the hard instance from the Euclidean case. Then we define
\begin{equation}\label{eq:embedded_hard_instance}
    F(X) \coloneqq \sum_{i=1}^\kappa F^{(2)}(P_i(X)), \qquad
    \nabla f(X, \xi) \coloneqq \sum_{i=1}^\kappa E_i\pare{\nabla f^{(2)}(P_i(X), \xi_i)},
\end{equation}
where $\xi_i \iid \operatorname{Bernoulli}(\theta)$ are independent. Next we prove that this embedding --- after appropriate rescaling --- preserves the required properties.

\begin{proposition}\label{prop:satisfies_properties}
    Let $\Dz^{(2)} = \frac{\Dz}{\kappa}, L^{(2)} \coloneqq \frac L \kappa, \eps^{(2)} \coloneqq \frac \eps \kappa, \varsym^{(2)} \coloneqq \frac{\varsym}{\kappa^{1/p}}$ and $F, \nabla f$ be defined as above. Furthermore let $c \coloneqq 1 / 300$ and assume $0 < \eps \leq c \min \set{\varsym, \sqrt{\Dz L}}$. Then $F \in \Fc_{\Dz, L}(m, n)$ and $\nabla f \in \Oc_{p, \varsym}(m, n)$. 
\end{proposition}
\begin{proof}
    It is straight-forward to check $\eps^{(2)} \leq c \min \set{\varsym^{(2)}, \sqrt{\Dz^{(2)} L^{(2)}}}$, which we require to apply \Cref{lem:eucl_lower_bound}. Hence the Euclidean hard instance exists and we calculate
    \begin{equation*}
        F(0) - \inf_{X} F(X)
        = \sum_{i=1}^\kappa F^{(2)}(0) - \inf_{x} F^{(2)}(x)
        \leq \kappa \Dz^{(2)} = \Dz,
    \end{equation*}
    and, for $X, Y \in \R^{m \times n}$,
    \begin{align*}
        \nucnorm{\nabla F(X) - \nabla F(Y)}
        =&\  \nucnorm{\sum_{i=1}^\kappa E_i\pare{\nabla F^{(2)}(P_i(X)) - \nabla F^{(2)}(P_i(Y))}}\\
        \stackAlign{\Cref{lem:embedding_properties}}{\leq} \sum_{i=1}^\kappa \norm{\nabla F^{(2)}(P_i(X)) - \nabla F^{(2)}(P_i(Y))}_2\\
        \leq&\  L^{(2)} \sum_{i=1}^\kappa \norm{P_i(X) - P_i(Y)}_2\\
        \stackAlign{\Cref{lem:embedding_properties}}{\leq} \ L^{(2)} \kappa \opnorm{X - Y} \\
        = &\ L \opnorm{X - Y}.
    \end{align*}
    This shows $F \in \Fc_{\Dz, L}(m, n)$. For the oracle properties, note that 
    \begin{equation*}
        \Exp{\nabla f(X, \xi)} = \sum_{i=1}^\kappa E_i\pare{\Exp{\nabla f^{(2)}(P_i(X), \xi_i)}} = \sum_{i=1}^\kappa E_i\pare{\nabla F^{(2)}(P_i(X))} = \nabla F(X).
    \end{equation*}
    To prove the bounded $p$-th central moment, we make use of the explicit structure of $\nabla f(X, \xi)$ to get
    \begin{align}\label{eq:satisfies_props_1}\begin{split}
        \Exp{\nucnorm{\nabla F(X) - \nabla f(X, \xi)}^p}
        =&\ \Exp{\nucnorm{\sum_{i=1}^\kappa E_i\pare{\nabla F^{(2)}(P_i(X)) - \nabla f^{(2)}(P_i(X), \xi_i)}}^p}\\
        =&\ \Exp{\nucnorm{\sum_{i=1}^\kappa E_i\pare{\pare{1 - \frac{\xi_i}{\theta}}h(P_i(X))}}^p}.
    \end{split}\end{align}
    By \Cref{lem:embedding_properties} we have 
    \begin{equation}\label{eq:satisfies_props_2}
        \nucnorm{\sum_{i=1}^\kappa E_i\pare{\pare{1 - \frac{\xi_i}{\theta}}h(P_i(X))}}
        = \sum_{i=1}^\kappa \abs{1 - \frac{\xi_i}{\theta}} \norm{h(P_i(X))}_2
        \leq 46 \eps^{(2)} \sum_{i=1}^\kappa \abs{1 - \frac{\xi_i}{\theta}}.
    \end{equation}
    Now let $N \coloneqq \sum_{i=1}^\kappa \xi_i \sim \operatorname{Binomial}(\kappa, \theta)$, and 
    \begin{equation}\label{eq:satisfies_props_q_def}
        q \coloneqq \kappa \theta = \kappa \pare{\frac{\eps^{(2)}}{c \varsym^{(2)}}}^\frac p {p-1} = \pare{\frac{\eps}{c \varsym}}^{\frac p {p-1}} \leq 1
    \end{equation}
    by our assumption $\eps \leq c \varsym$. Next,
    \begin{equation}\label{eq:satisfies_props_3}
        \frac 1 \kappa \sum_{i = 1}^\kappa \abs{1 - \frac{\xi_i}{\theta}}
        \leq 1 + \frac{\sum_{i=1}^\kappa \xi_i}{\kappa \theta}
        = 1 + \frac{N}{q}
    \end{equation}
    and hence
    \begin{align*}
        \Exp{\nucnorm{\nabla F(X) - \nabla f(X, \xi)}^p}
        \stackAlign{\eqref{eq:satisfies_props_1}, \eqref{eq:satisfies_props_2}}{\leq} \pare{\frac{46 \eps}\kappa}^p \Exp{\pare{\sum_{i=1}^\kappa \abs{1 - \frac{\xi_i}{\theta}}}^p}\\
        \stackAlign{\eqref{eq:satisfies_props_3}}{\leq} \pare{46 \eps}^p \Exp{\pare{1 + \frac{N}{q}}^p}\\
        \leq&\ 2 \pare{46 \eps}^p \pare{1 + \Exp{\pare{\frac{N^p}{q^p}}}}.
    \end{align*}
    Finally we calculate
    \begin{equation*}
        \Exp{N^p}
        \leq \Exp{N^2}
        = \kappa \theta (1 - \theta) + (\kappa \theta)^2
        = q (1 - \theta) + q^2
        \leq 2 q,
    \end{equation*}
    which yields
    \begin{equation*}
        2 \pare{46 \eps}^p \pare{1 + \Exp{\pare{\frac{N^p}{q^p}}}}
        \leq 2 \pare{46 \eps}^p \pare{1 + 2 q^{1-p}}
        \leq 6 \pare{46 c}^p \varsym^p
        \leq \varsym^p,
    \end{equation*}
    and hence the claim.
\end{proof}

Now we are ready to prove the first-order lower-bound for $\nucnorm \cdot$-stationarity.

\begin{theorem}[First-Order Lower-Bound for $\nucnorm \cdot$-Stationarity]\label{thm:high_dim_lower_bound}
    Let $c \coloneqq 1 / 300, \kappa \in \Ngeq$ and $p \in (1,2]$. For all $\Dz, L, \varsym > 0$ and $0 < \eps \leq c \min\set{\varsym, \sqrt{\Dz L}}$ there exists $m \leq \kappa \frac{\Dz L}{\eps^2}$, a function $F \in \Fc_{\Dz, L}(m, \kappa)$ and a gradient oracle $\nabla f \in \Oc_{p, \varsym}(m, \kappa)$ such that every possibly randomised zero-respecting algorithm $\Ac$ satisfies
    \begin{equation*}
        \Exp{\nucnorm{\nabla F(x_t)}} > \eps \qquad \text{for all} \qquad t \leq c^2 \cdot \min\set{m, n} \frac{\Dz L}{\eps^2}\pare{\frac{c \varsym}{\eps}}^{\frac{p}{p-1}}.
    \end{equation*}
    In particular, there does not exist a dimension-free first-order guarantee in $\nucnorm \cdot$-stationarity, and \muon\ is first-order optimal.
\end{theorem}
\begin{proof}
    Let $d \leq \Dz L \eps^{-2}$ be the dimension from \Cref{lem:eucl_lower_bound} and set $m \coloneqq \kappa d$. Furthermore let $F, \nabla f$ be the hard instance defined in \Cref{eq:embedded_hard_instance}. By \Cref{prop:satisfies_properties} we have $F \in \Fc_{\Dz, L}(m, \kappa)$ and $\nabla f \in \Oc_{p, \varsym}(m, \kappa)$.
    
    Now let $\Ac$ be a zero-respecting algorithm and denote the sequence generated by $\Ac$ on $F, \nabla f$ as $(X_t)_{t \in \Ngeq}$. Let $i \in [\kappa]$ and note that, by the independence of $\xi_1, \dots, \xi_\kappa$, the sequence $x^{(i)}_t \coloneqq P_i(X_t) \in \R^d$ corresponds to a randomised zero-respecting algorithm on $F^{(2)}, \nabla f^{(2)}$. In particular, we have
    \begin{equation*}
        \Exp{\norm{\nabla F^{(2)}\pare{x_t^{(i))}}}_2} > \eps^{(2)}
        \quad \text{for all} \quad t \leq c^2 \cdot \frac{\Dz^{(2)} L^{(2)}}{(\eps^{(2)})^2}\pare{\frac{c \varsym^{(2)}}{\eps^{(2)}}}^{\frac{p}{p-1}}.
    \end{equation*}
    with our choices $\Dz^{(2)} = \frac{\Dz}{\kappa}, L^{(2)} \coloneqq \frac L \kappa, \eps^{(2)} \coloneqq \frac \eps \kappa, \varsym^{(2)} \coloneqq \frac{\varsym}{\kappa^{1/p}}$ this is equivalent to
    \begin{equation*}        
        \Exp{\norm{\nabla F^{(2)}\pare{x_t^{(i))}} }_2} > \frac \eps \kappa
        \quad \text{for all} \quad t \leq c^2 \cdot \kappa \frac{\Dz L}{\eps^2}\pare{\frac{c \varsym}{\eps}}^{\frac{p}{p-1}}.
    \end{equation*}
    Finally, by \Cref{lem:embedding_properties} we have
    \begin{equation*}
        \Exp{\nucnorm{\nabla F(X_t)}}
        = \Exp{\nucnorm{\sum_{i = 1}^\kappa E_i\pare{\nabla F^{(2)}\pare{x_t^{(i))}}}}}
        = \sum_{i=1}^\kappa \Exp{\norm{\nabla F^{(2)}\pare{x_t^{(i))}}}_2}
        > \eps,
    \end{equation*}
    which finishes the proof.
\end{proof}

\Cref{thm:high_dim_lower_bound} is a direct corollary of the previous result, by noting that the case $m < n$ can be handled by considering the transposed version of the above proof, and unnecessary dimensions can be padded with zeros.

\section{Experimental Details}
\label{sec:app.experimental_details}
In this section we provide further details and extensions of our experiments in \Cref{sec:experiments}.

\paragraph{Further Details.} We follow the usual convention of training all 2-d parameters but the input- and output-embeddings with \muon, while all other parameters are trained with \textsc{AdamW}.

\paragraph{Computational Resources.} All experiments were conducted on an internal cluster, using 2 H100 GPUs. Overall the experiments required 606 GPU hours, out of which 311 were spent on preliminary experiments, 283 on the Schatten-r experiments, and an additional 12 GPU hours to calculate the gradient noise. 

\paragraph{Licenses.} The FineWeb-Edu dataset \citep{penedo2024fineweb} is distributed under the Open Data Commons Attribution License (ODC-By) v1.0. Our use complies with this license and adheres to CommonCrawl’s Terms of Use. 

\paragraph{Details on Figure \texorpdfstring{\ref{fig:noise_ratios_layers} and \ref{fig:embedding_plots}}{1 and 2}.}
We consider the checkpoint at initialisation and at the end of training for the optimal learning rate $\eta = 0.015$, and seed 0. For both checkpoints, we first calculate the true gradient $\nabla F(x_i), i \in \set{1, T}$ by going over the whole $1.4$B token training dataset $\{\xi_1, \dots, \xi_T\}$. In a second pass we then calculate, for each weight matrix $x^{(j)}$
\begin{align*}
    \sigma_{S_1, p}^p &= \frac 1 T \sum_{\tau = 1}^{T}
    \nucnorm{\partial_{x^{(j)}} F(x_i) - \partial_{x^{(j)}} f(x_i, \xi_\tau)}^p,\\
    \sigma_{F, p}^p &= \frac 1 T \sum_{\tau = 1}^{T}
    \fnorm{\partial_{x^{(j)}} F(x_i) - \partial_{x^{(j)}} f(x_i, \xi_\tau)}^p,
\end{align*}
where $\iterationxi \fin, \dots, \iterationxi \li$ are the mini-batches of size 512. We then calculate the ratio $\nicefrac{\sigma_{S_1, p}}{\sigma_{F, p}}$ for all weight matrices trained by \muon\ and report them in \Cref{fig:noise_ratios_layers}. We additionally visualise the plain $\sigma_{S_1, p}$ and $\sigma_{\ell_1, p}$ in \Cref{fig:plain_sigma}.

\begin{figure}
    \centering
    
    \begin{subfigure}[t]{0.48\textwidth}
        \centering
        \includegraphics[width=\textwidth]{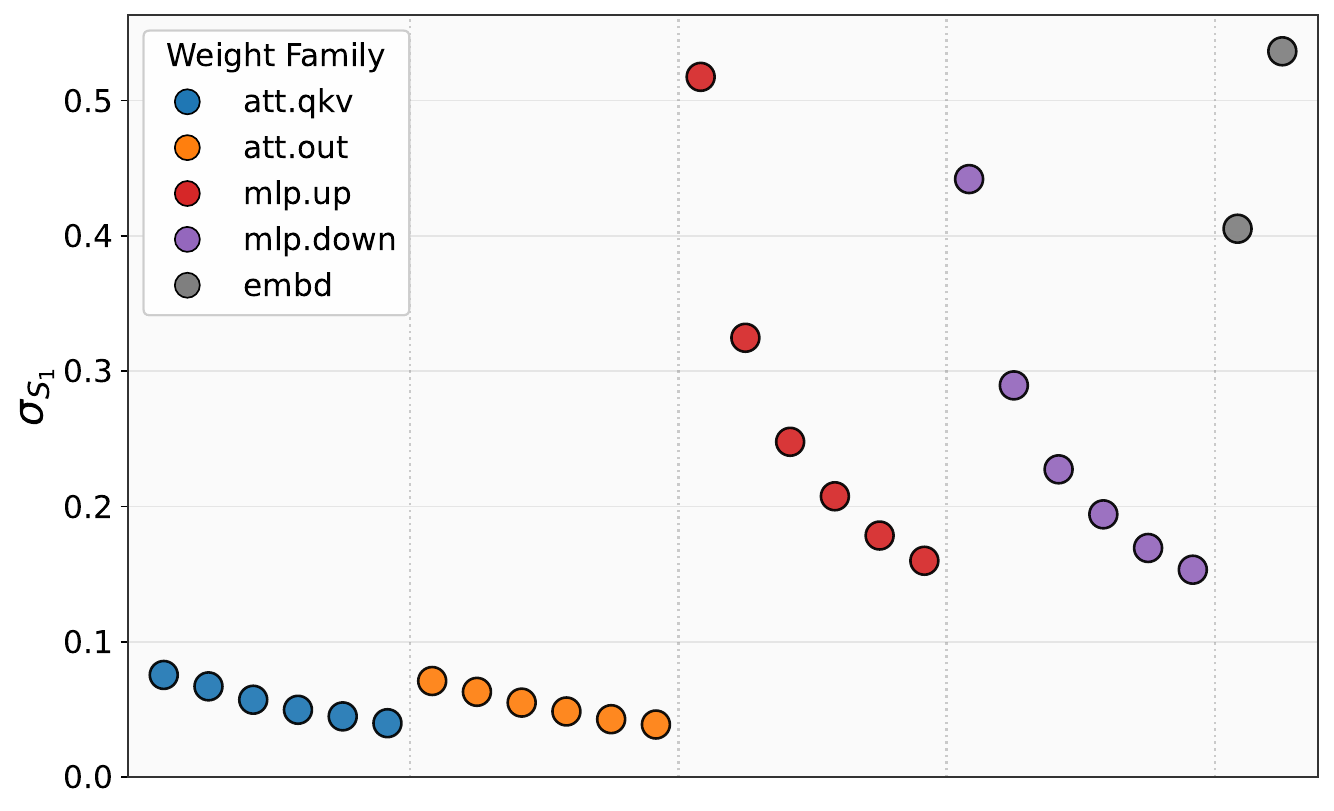}
        \caption{$\sigma_{S_1}$ at Init}
    \end{subfigure}
    \hfill
    \begin{subfigure}[t]{0.48\textwidth}
        \centering
        \includegraphics[width=\textwidth]{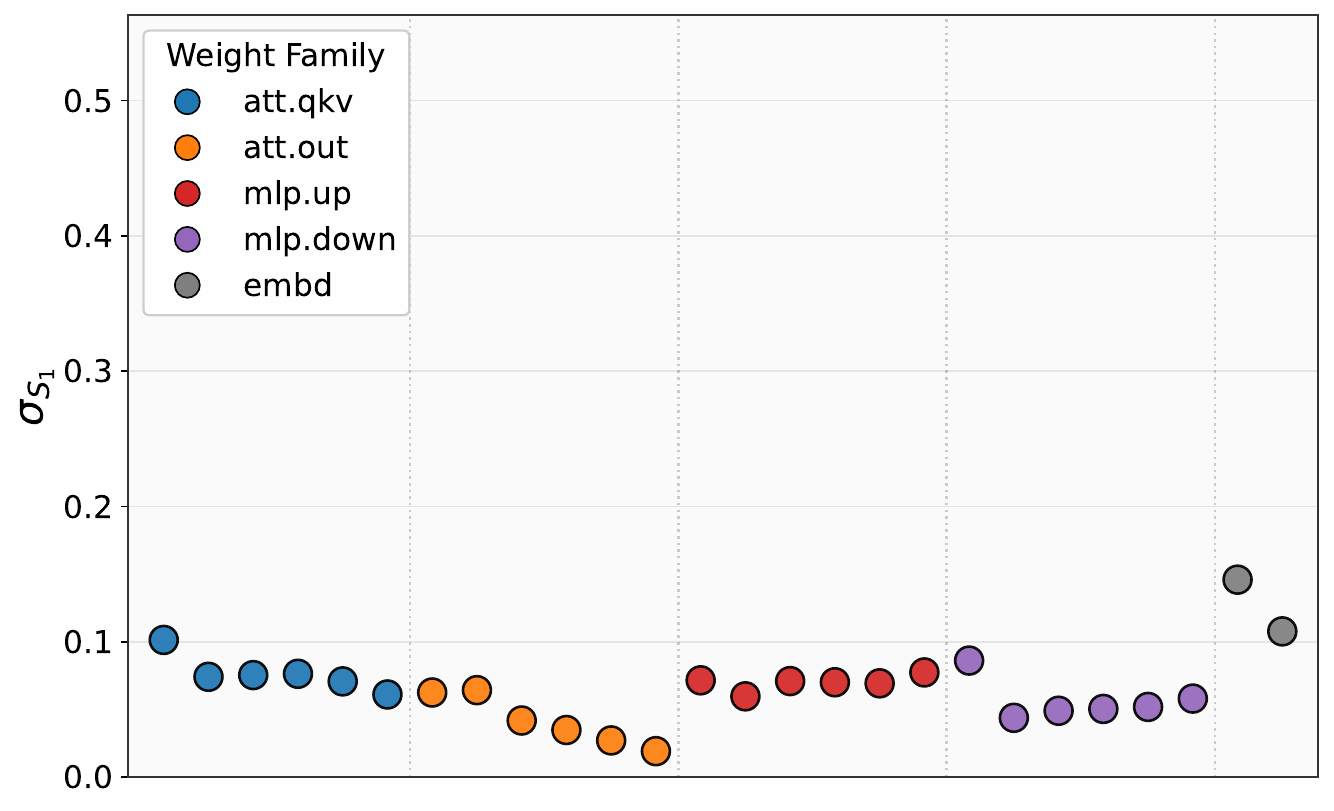}
        \caption{$\sigma_{S_1}$ at Final}
    \end{subfigure}
    
    \vspace{0.5em}

    \begin{subfigure}[t]{0.48\textwidth}
        \centering
        \includegraphics[width=\textwidth]{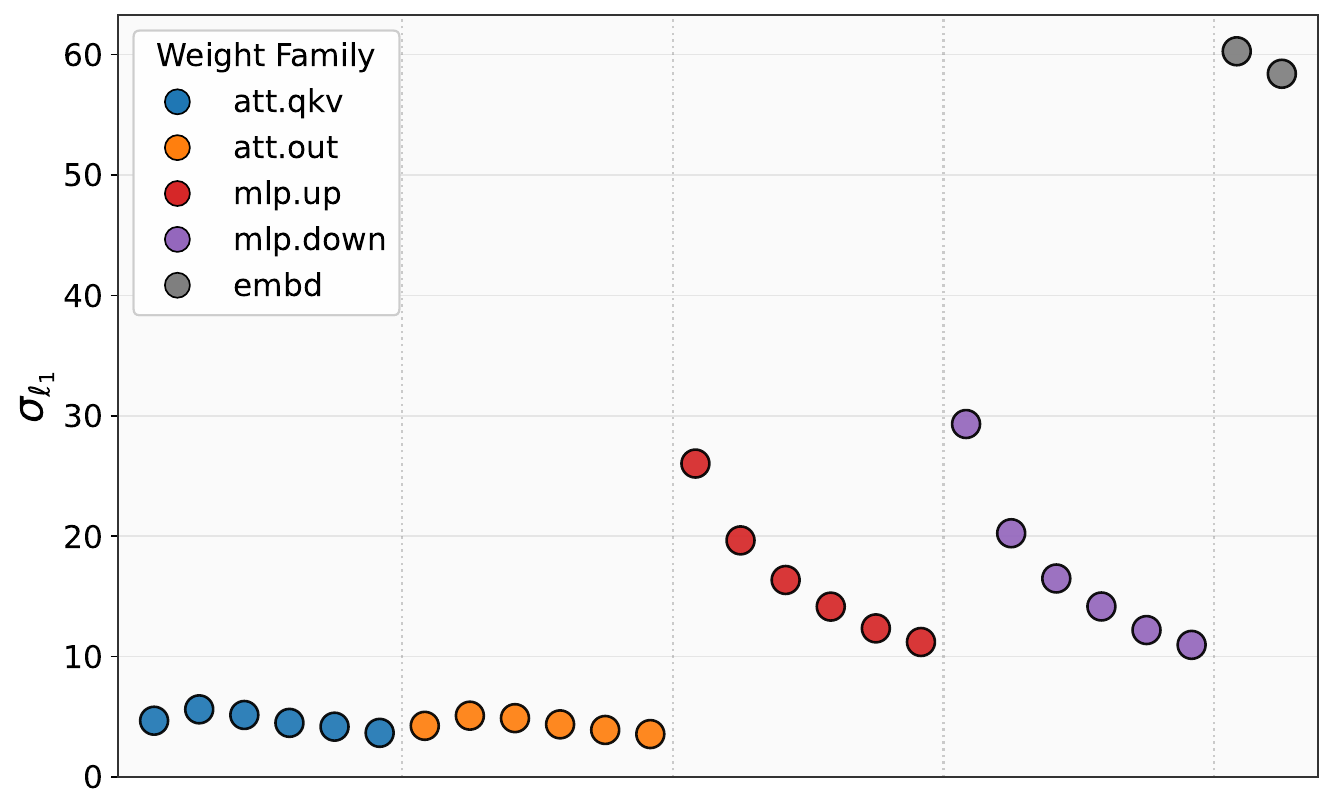}
        \caption{$\sigma_{\ell_1}$ at Init}
    \end{subfigure}
    \hfill
    \begin{subfigure}[t]{0.48\textwidth}
        \centering
        \includegraphics[width=\textwidth]{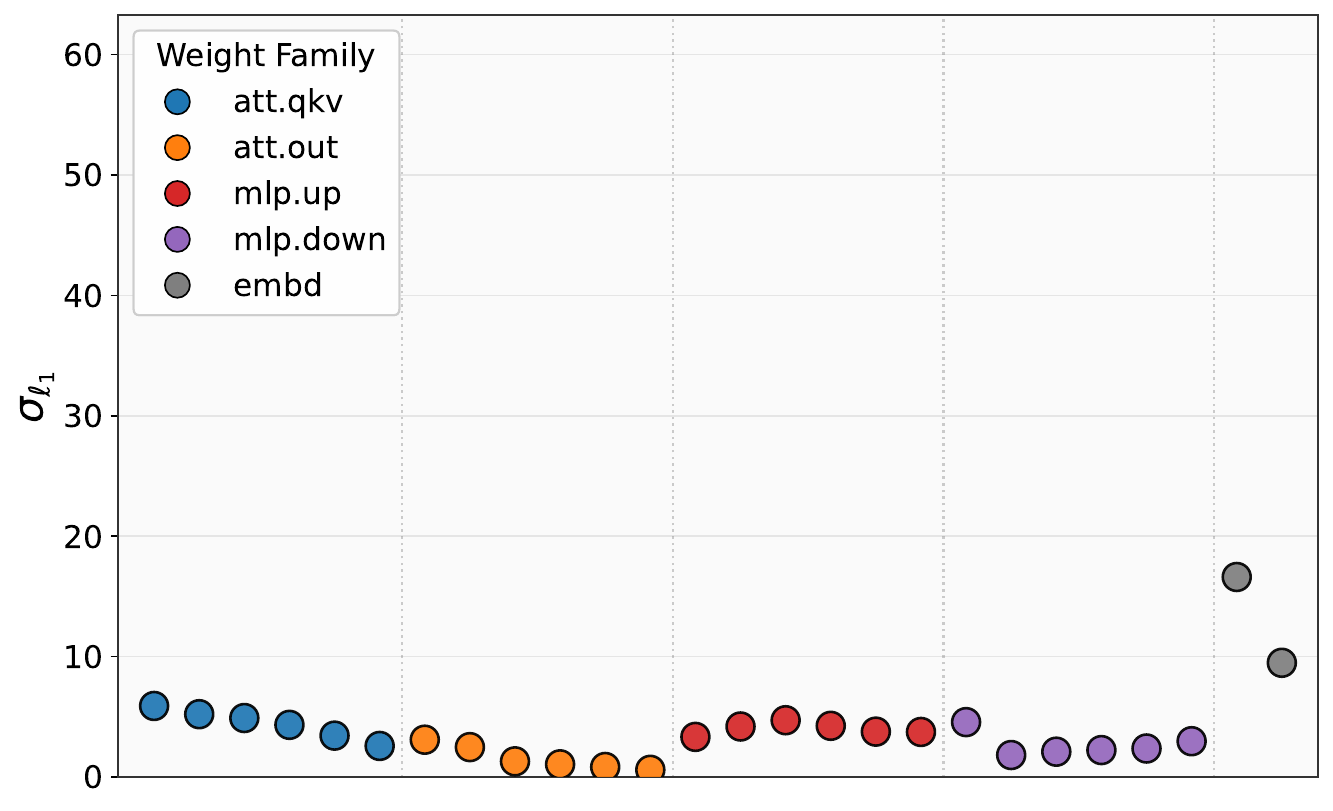}
        \caption{$\sigma_{\ell_1}$ at Final}
    \end{subfigure}

    \caption{Plain noise values.}
    \label{fig:plain_sigma}
\end{figure}

\paragraph{Details on \Cref{tab:schatten_comparison}.} To prevent confounding from the choice of learning rates, we additionally provide the full grid of results for the learning rate and Schatten-$r$ sweep in \Cref{fig:sweep_grid}. Importantly, none of the optimal learning rates lay on the end points of the grid.

\begin{figure}
    \centering
    \includegraphics[width=\textwidth]{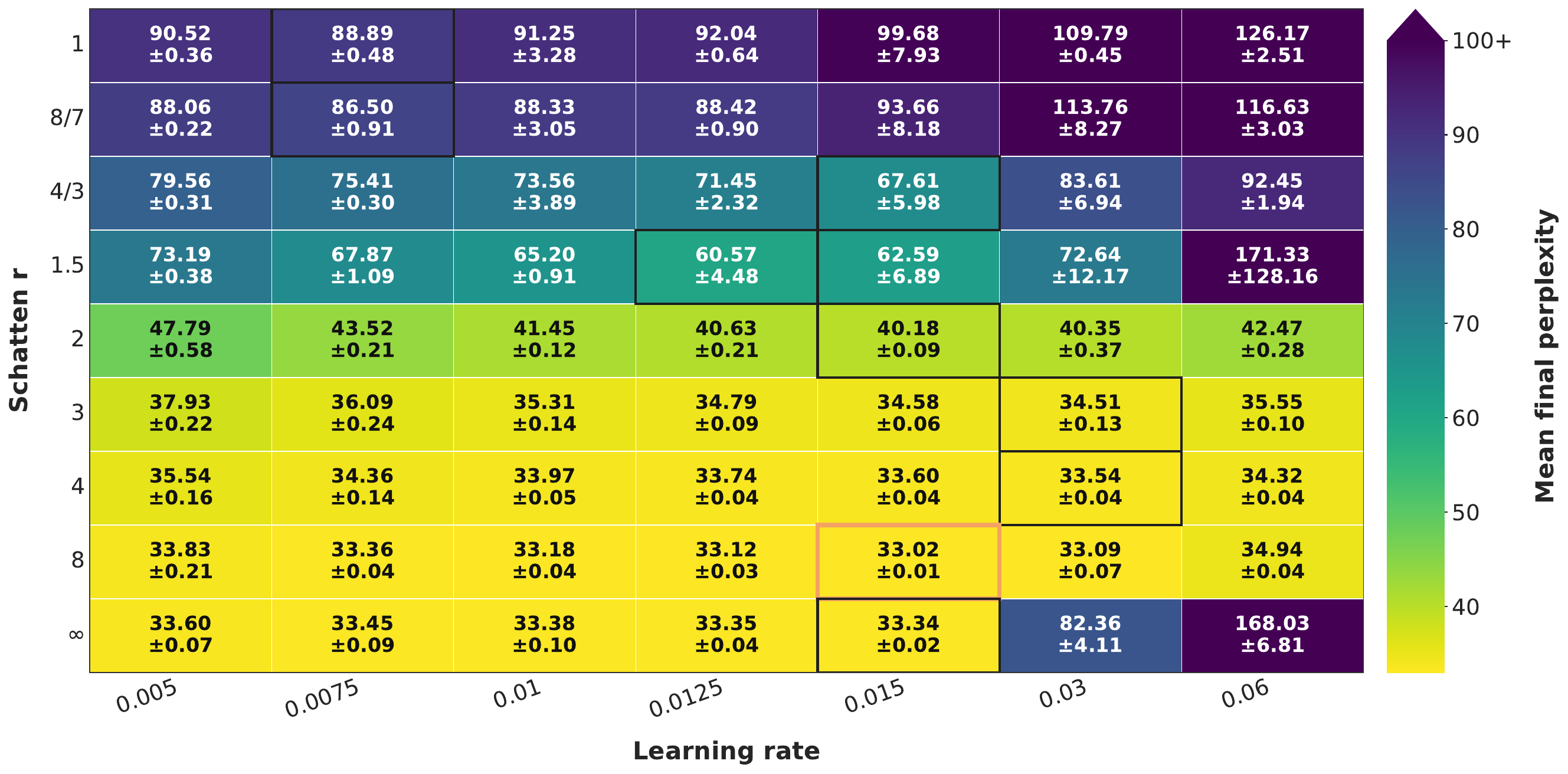}
    \caption{Full result grid for the Schatten-$r$ sweep. The value represents the mean final perplexity $\pm$ its standard deviation over the seeds $0, 1, 2$. Optimal learning rates per geometry are marked with black borders, the overall best result with an orange border.}
    \label{fig:sweep_grid}
\end{figure}

\clearpage

\end{document}